\documentclass[12pt,a4paper]{article}%
\usepackage[utf8]{inputenc}
\usepackage{amsmath}
\usepackage{amsfonts}
\usepackage{amssymb}
\usepackage{graphicx}
\usepackage[space]{grffile}
\usepackage{hyperref}%
\providecommand{\U}[1]{\protect\rule{.1in}{.1in}}
\newtheorem{theorem}{Theorem}
\newtheorem{theorem*}{Example}

\newtheorem{conjecture}[theorem]{Conjecture}
\newtheorem{corollary}[theorem]{Corollary}

\newtheorem{definition}[theorem]{Definition}

\newtheorem{proposition}[theorem]{Proposition}
\newtheorem{remark}[theorem]{Observation}

\newenvironment{proof}[1][Proof]{\noindent\textbf{#1.} }{\ \hfill \rule{0.5em}{0.5em}\bigskip}
\graphicspath{{D:/Dropbox/Riste-Sedlar/Edge colorings/Paper1/Slike/}}
\begin{document}

\title{Normal $5$-edge-coloring of some snarks superpositioned by the Petersen graph}
\author{Jelena Sedlar$^{1,3}$,\\Riste \v Skrekovski$^{2,3}$ \\[0.3cm] {\small $^{1}$ \textit{University of Split, Faculty of civil
engineering, architecture and geodesy, Croatia}}\\[0.1cm] {\small $^{2}$ \textit{University of Ljubljana, FMF, 1000 Ljubljana,
Slovenia }}\\[0.1cm] {\small $^{3}$ \textit{Faculty of Information Studies, 8000 Novo
Mesto, Slovenia }}\\[0.1cm] }
\maketitle

\begin{abstract}
In a (proper) edge-coloring of a bridgeless cubic graph $G$ an edge $e$ is
rich (resp. poor) if the number of colors of all edges incident to
end-vertices of $e$ is $5$ (resp. $3$). An edge-coloring of $G$ is is normal
if every edge of $G$ is either rich or poor. In this paper we consider snarks
$\tilde{G}$ obtained by a simple superposition of edges and vertices of a
cycle $C$ in a snark $G.$ For an even cycle $C$ we show that a normal coloring
of $G$ can be extended to a normal coloring of $\tilde{G}$ without changing
colors of edges outside $C$ in $G.$ An interesting remark is that this is in
general impossible for odd cycles, since the normal coloring of a Petersen
graph $P_{10}$ cannot be extended to a superposition of $P_{10}$ on a
$5$-cycle without changing colors outside the $5$-cycle. On the other hand, as
our colorings of the superpositioned snarks introduce $18$ or more poor edges,
we are inclined to believe that every bridgeless cubic graph distinct from
$P_{10}$ has a normal coloring with at least one poor edge and possibly with
at least $6$ if we also exclude the Petersen graph with one vertex truncated.

\end{abstract}

\textit{Keywords:} normal edge-coloring; cubic graph; snark; superposition;
Petersen Coloring Conjecture.

\textit{AMS Subject Classification numbers:} 05C15

\section{Introduction}

All graphs we consider in this paper are tacitly assumed to be simple, cubic
and bridgeless, unless explicitly stated otherwise. Let $G$ be a graph with
the set of vertices $V(G)$ and the set of edges $E(G).$ For a vertex $v\in
V(G)$ we denote by $\partial_{G}(v)$ the set of edges in $G$ which are
incident to $v.$ Let $G$ and $H$ be a pair of cubic graphs. A $H$%
\emph{-coloring} of $G$ is any function $\phi:E(G)\rightarrow E(H)$ where for
every $v\in V(G)$ there exists $w\in V(H)$ for which $\phi(\partial
_{G}(v))=\partial_{H}(w).$ If $G$ has an $H$-coloring, we denote it by $H\prec
G.$

Let us denote by $P_{10}$ the Petersen graph. The following conjecture was
stated by Jaeger \cite{Jaeger1988}.

\begin{conjecture}
[Petersen Coloring Conjecture]\label{Con_petersen}For a bridgeless cubic graph
$G,$ it holds that $P_{10}\prec G.$
\end{conjecture}

It is known that the Petersen Coloring Conjecture implies several well known
conjectures, such as Berge-Fulkerson Conjecture and also (5,2)-cycle-cover
Conjecture, thus it seems difficult to prove. It motivated several variants
and different approaches to it \cite{Mkrtchjan2013, Riste2020, Samal2011}, one
such approach are normal colorings.

\paragraph{Normal colorings.}

A \emph{(proper) }$k$\emph{-edge-coloring} of a graph $G$ is any mapping
$\sigma:E(G)\rightarrow\{1,\ldots,k\}$ such that any pair of adjacent edges of
$G$ receives distinct colors by $\sigma$. Let $\sigma$ be a $k$-edge-coloring
of $G$ and $v\in V(G),$ then we denote $\sigma(v)=\{\sigma(e):e\in\partial
_{G}(v)\}.$

\begin{definition}
Let $G$ be a bridgeless cubic graph and $\sigma$ a proper edge-coloring of
$G.$ An edge $uv$ of $G$ is called \emph{poor} (resp. \emph{rich}) if
$\left\vert \sigma(u)\cup\sigma(v)\right\vert =3$ (resp. $\left\vert
\sigma(u)\cup\sigma(v)\right\vert =5$).
\end{definition}

An edge of $G$ is called \emph{normal}, if it is either rich or poor. A proper
edge-coloring of a cubic graph $G$ is \emph{normal}, if every edge of $G$ is
normal. Normal edge-colorings were introduced by Jaeger in \cite{Jaeger1985}.
Notice that a normal coloring in which every edge is poor must be a
$3$-edge-coloring. On the other hand, a normal coloring in which every edge is
rich is called a \emph{strong edge-coloring}. The \emph{normal chromatic
index} of a cubic graph $G$ is defined as the smallest number $k$ such that
$G$ has a normal $k$-edge-coloring, and it is denoted by $\chi_{N}^{\prime
}(G).$ Notice that $\chi_{N}^{\prime}(G)$ is at least $3$ and it never equals
$4$. The following result was established in \cite{Jaeger1985}.

\begin{proposition}
\label{Prop_equivalentNormalPetersen}For a cubic graph $G$, it holds that
$P_{10}\prec G$ if and only if $G$ has a normal $5$-edge-coloring.
\end{proposition}

According to Proposition \ref{Prop_equivalentNormalPetersen}, Conjecture
\ref{Con_petersen} can be restated as follows.

\begin{conjecture}
\label{Con_normal}For a bridgeless cubic graph $G,$ it holds that $\chi
_{N}^{\prime}(G)\leq5.$
\end{conjecture}

\noindent Since any proper $3$-edge-coloring of a cubic graph $G$ is a normal
edge-coloring in which every edge is poor, Conjecture \ref{Con_normal}
obviously holds for every $3$-edge colorable graph. The best result so far
regarding the upper bound on the normal chromatic index is $\chi_{N}^{\prime
}(G)\leq7$ \cite{BilkovaHana, Mazzucuolo2020normal}, and for several graph
classes the bound is lowered to $6$ \cite{Mazzucuolo2020normal6}. According to
Vizing's theorem, every cubic graph is either $3$-edge colorable or $4$-edge
colorable, so to prove Conjecture \ref{Con_normal} it remains to establish
that it holds for all bridgeless cubic graphs which are not $3$-edge colorable.

\paragraph{Snarks and superposition.}

Cubic graphs which are not $3$-edge colorable are usually considered under the
name of snark \cite{MazzuoccoloStephen,NedelaSkovieraSurvey}. In order to
avoid trivial cases, the definition of snark usually includes some additional
requirements on connectivity. Since in this paper such requirements are not
essential, we will go with the broad definition of a \emph{snark} being any
bridgeless cubic graph which is not $3$-edge colorable. The existence of a
normal $5$-coloring for some families of snarks has already been established,
see for example \cite{MazzuccoloLupekhine, Hagglund2014}.

A general method for obtaining snarks is the superposition \cite{Adelson1973,
Descartes1946, KocholSuperposition, Kochol2, MacajovaRevisited}. We will focus
our attention to a class of snarks obtained by some particular superpositions
of edges and vertices of a cycle in a snark, so let us first introduce the
method of superposition.

\begin{definition}
A \emph{multipole} $M=(V,E,S)$ is a triple which consists of a set of vertices
$V=V(M)$, a set of edges $E=E(M)$, and a set of semiedges $S=S(M)$. A semiedge
is incident either to one vertex or to another semiedge in which case a pair
of incident semiedges forms a so called \emph{isolated edge} within the multipole.
\end{definition}

\begin{figure}[h]
\begin{center}%
\begin{tabular}
[c]{ll}%
\begin{tabular}
[c]{c}%
\includegraphics[scale=0.6]{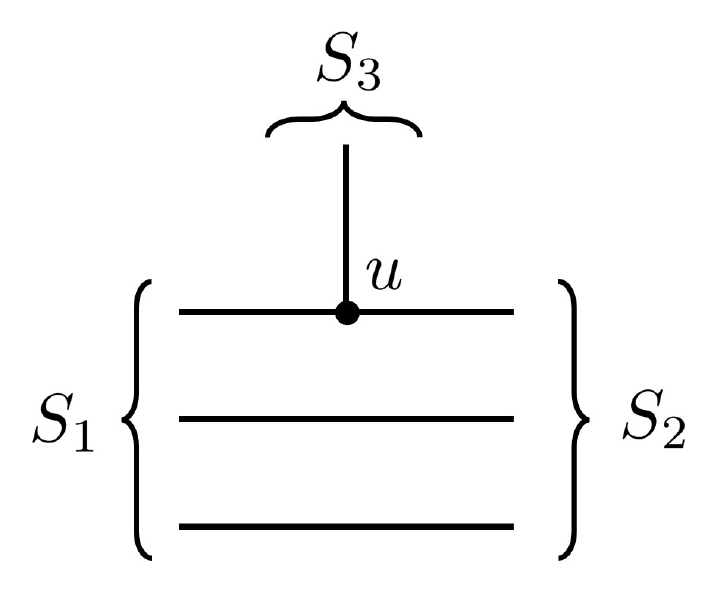}\\
$A$%
\end{tabular}
&
\begin{tabular}
[c]{c}%
\includegraphics[scale=0.6]{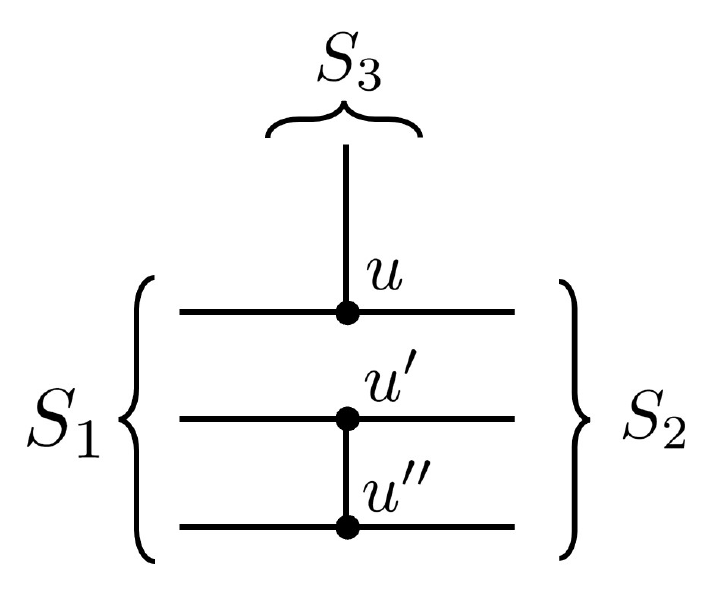}\\
$A^{\prime}$%
\end{tabular}
\end{tabular}
\end{center}
\caption{Multipoles $A$ and $A^{\prime}$. Notice that $A$ has two isolated
edges. Both $A$ and $A^{\prime}$ are $7$-poles, and if their set of semiedges
is divided into three connectors $S_{1},$ $S_{2}$ and $S_{3}$ as shown in the
figure, then $A$ and $A^{\prime}$ are supervertices.}%
\label{Fig_superVert}%
\end{figure}

For an illustration of multipole see Figure \ref{Fig_superVert}. Notice that a
graph $G$ is a multipole with $S=\emptyset.$ If $\left\vert S(M)\right\vert
=k,$ then $M$ is also called a $k$\emph{-pole}. All multipoles we consider
here are cubic, i.e. such that every vertex is incident to precisely three
edges or semiedges. The set of semiedges of a $k$-pole $M$ can be divided into
several so called connectors, i.e. we define a $(k_{1},\ldots,k_{n}%
)$\emph{-pole} to be $M=(V,E,S_{1},\ldots,S_{n})$ where $S_{1},\ldots,S_{n}$
is a partition of the set $S(M)$ into pairwise disjoint sets $S_{i}$ such that
$\left\vert S_{i}\right\vert =k_{i}>0$ for $i=1,\ldots,n$ and $k_{1}%
+\cdots+k_{n}=k.$ Each set $S_{i}$ is called a \emph{connector} of $M,$ or
more precisely $k_{i}$\emph{-connector}.

\begin{definition}
A \emph{supervertex} (resp. \emph{superedge}) is a multipole with three (resp.
two) connectors.
\end{definition}

Let $G$ be a snark and $u_{1},u_{2}$ a pair of non-adjacent vertices of $G.$ A
superedge $G_{u_{1},u_{2}}$ is obtained from $G$ by removing vertices $u_{1}$
and $u_{2},$ and then replacing all edges incident to a vertex $u_{i}$ in $G$
by semiedges in $G_{u_{1},u_{2}}$ which form a connector $S_{i}$, for $i=1,2.$

\begin{definition}
A \emph{proper }superedge is any superedge $G_{u_{1},u_{2}},$ where $G$ is a
snark, or an isolated edge.
\end{definition}

\noindent A proper superedge can be defined in a wider sense (see
\cite{KocholSuperposition}), but for the purposes of the present paper this
simple definition suffices. Now, let $G=(V,E)$ be a cubic graph. Replace each
edge $e\in E$ by a superedge $\mathcal{E}(e)$ and each vertex $v\in V$ by a
supervertex $\mathcal{V}(v)$ so that the following holds: if $v\in V$ is
incident to $e\in E$, then semiedges of one connector of $\mathcal{V}(v)$ are
identified with semiedges of one connector from $\mathcal{E}(e)$ and these two
connectors must have the same cardinality. Notice that $\mathcal{V}(v)$ (resp.
$\mathcal{E}(e)$) does not have to be the same for every vertex $v$ (resp.
edge $e),$ so $\mathcal{V}$ (resp. $\mathcal{E}$) is a mapping of the set of
vertices $V(G)$ to a set of supervertices (resp. superedges). This procedure
results in a cubic graph which is called a \emph{superposition} of $G$ with
$\mathcal{V}$ and $\mathcal{E}$ and denoted by $G(\mathcal{V},\mathcal{E})$.
Moreover, if $\mathcal{E}(e)$ is proper for every $e\in E$, then
$G(\mathcal{V},\mathcal{E})$ is called a \emph{proper} superposition of $G$.
Since a multipole consisting of a single vertex and three semiedges is a
supervertex, and an isolated edge is proper superedge, some of the vertices
and edges of $G$ may be superposed by themselves. The following theorem is
established in \cite{KocholSuperposition}, mind that it is stated there for
snarks of girth $\geq5$ which are cyclically $4$-edge connected, but it also
holds for snarks of smaller girth.

\begin{theorem}
\label{Tm_Kochol}If $G$ is a snark and $G(\mathcal{V},\mathcal{E})$ is a
proper superposition of $G$, then $G(\mathcal{V},\mathcal{E})$ is a snark.
\end{theorem}

We will establish that Conjecture \ref{Con_normal} holds for some snarks
$G(\mathcal{V},\mathcal{E})$ obtained for particular supervertices
$\mathcal{V}$ and superedges $\mathcal{E}$ superposed on vertices and edges of
a cycle in $G.$

\section{Main results}

Let us first extend the notion of normal colorings to multipoles. First, a
\emph{(proper) }$k$\emph{-edge-coloring} of a multipole $M=(V,E,S)$ is any
mapping $\sigma:E\cup S\rightarrow\{1,\ldots,k\}$ such that all edges and
semiedges incident to a same vertex have distinct colors. A proper
$k$-edge-coloring $\sigma$ of a cubic multipole $M$ is \emph{normal} if every
edge $e\in E$ of $M$ is normal. A handy tool when considering normal colorings
is the concept of Kempe chains. Let $\sigma$ be a normal coloring of a cubic
multipole $M=(V,E,S)$ and let $P=s_{1}e_{2}\cdots e_{p-1}s_{p}$, for $p>1,$
$e_{i}\in E$ and $s_{i}\in S,$ be a path in $M$ starting and ending with a
semiedge. We say that $P$ is a \emph{Kempe }$(i,j)$\emph{-chain} with respect
to $\sigma,$ if (semi)edges of $P$ are alternatively colored by $i$ and $j$.

\begin{remark}
\label{Obs_Kempe}Let $\sigma$ be a normal $k$-edge-coloring of a cubic
multipole $M=(V,E,S)$ and let $P=s_{1}e_{2}\cdots e_{p-1}s_{p}$ be a Kempe
$(i,j)$-chain of $M$ with respect to $\sigma.$ Then every edge $e_{i}$ from
$P$ is poor in $\sigma.$ Also, an edge-coloring $\sigma^{\prime}$ of $M$
obtained from $\sigma$ by swapping colors $i$ and $j$ along $P$ is also a
normal $k$-edge-coloring of $M$, preserving the same set of poor edges.
\end{remark}

\paragraph{Submultipoles.}

A \emph{submultipole} of a multipole $M=(V,E,S)$ is any multipole $M^{\prime
}=(V^{\prime},E^{\prime},S^{\prime})$ such that $V^{\prime}\subseteq V,$
$E^{\prime}\subseteq E$ and $S^{\prime}\subseteq S_{E}\cup S$ where $S_{E}$
denotes the set of semiedges obtained by halving the edges from $E.$ The
notion of submultipole is illustrated by Figure \ref{Fig_submultipole}.

\begin{figure}[h]
\begin{center}
\includegraphics[scale=0.7]{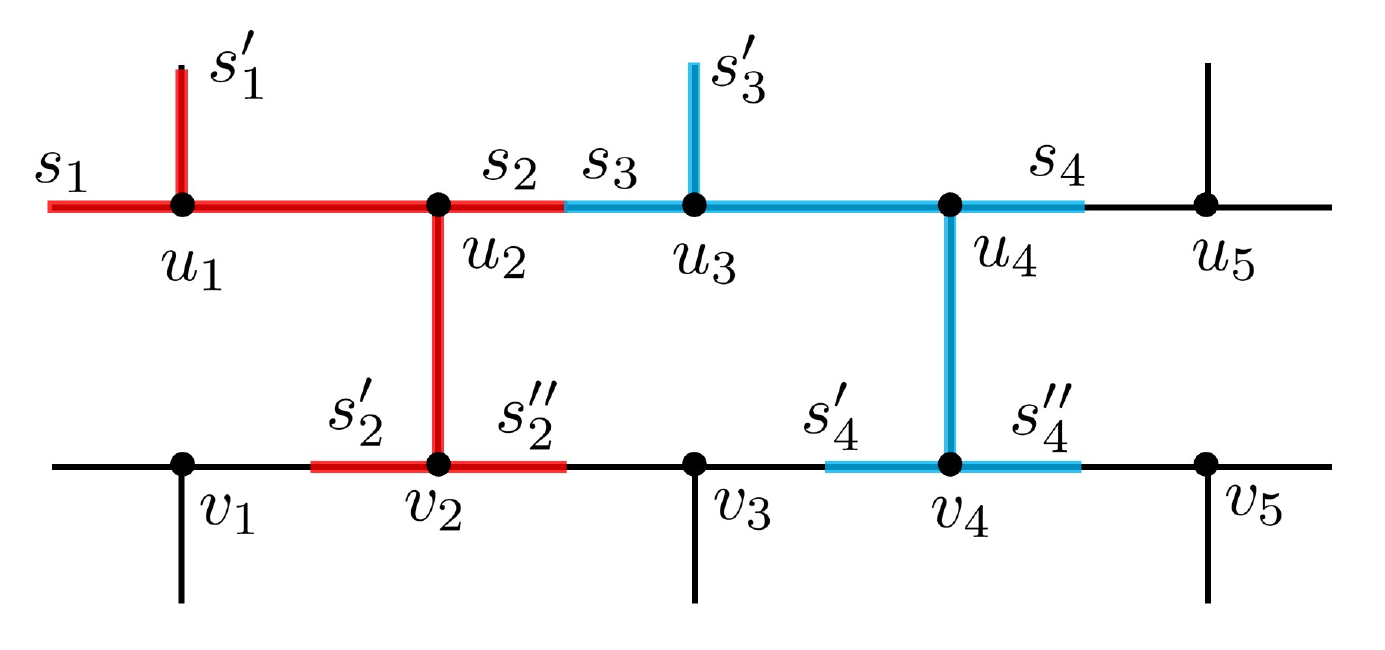}
\end{center}
\caption{Figure shows a multipole $M$ and two of its submultipoles,
$M^{\prime}$ in red and $M^{\prime\prime}$ in blue.}%
\label{Fig_submultipole}%
\end{figure}

\begin{remark}
Let $M$ be a multipole, and $M^{\prime}$ and $M^{\prime\prime}$ its two
submultipoles as in Figure \ref{Fig_submultipole}. Then semiedges $s_{1}$ and
$s_{1}^{\prime}$ of $M^{\prime}$ are inherited from the set of semiedges of
$M,$ on the other hand semiedges $s_{2},$ $s_{2}^{\prime}$ and $s_{2}%
^{\prime\prime}$ of $M^{\prime}$ are not inherited from the set of semiedges
of $M,$ they are the halves of edges from $M$ for which the other half is not
present in $M^{\prime}$.
\end{remark}

Let $\sigma$ be an edge-coloring of a multipole $M=(V,E,G)$ and let
$M^{\prime}=(V^{\prime},E^{\prime},G^{\prime}).$ A \emph{restriction}
$\sigma^{\prime}=\left.  \sigma\right\vert _{M^{\prime}}$ of the coloring
$\sigma$ to the submultipole $M^{\prime}$ is defined by $\sigma^{\prime
}(e)=\sigma(e)$ for every $e\in E^{\prime},$ $\sigma^{\prime}(s)=\sigma(s)$
for every $s\in S\cap S^{\prime}$ and $\sigma^{\prime}(s)=\sigma(e)$ for every
$s\in S^{\prime}\backslash S$ which is a half of an edge $e\in E.$ For a set
$T\subseteq V(M),$ an \emph{induced} submultipole $M^{\prime}=M[T]$ of $M$ is
a submultipole in which the set of vertices is $T,$ set of edges is the set of
all edges of $G$ with both end-vertices in $T,$ and the set of semiedges
consists of all semiedges of $M$ with its only end-vertex in $T$ and halves of
edges of $M$ which have precisely one end-vertex in $T.$ Let $M$ be a
multipole, and $M^{\prime}=(V^{\prime},E^{\prime},S^{\prime})$ and
$M^{\prime\prime}=(V^{\prime\prime},E^{\prime\prime},S^{\prime\prime})$ a pair
of submultipoles of $M.$ Let $V^{\prime\prime\prime}=V^{\prime}\cup
V^{\prime\prime},$ further let $E^{\prime\prime\prime}$ be the set of edges
$e$ of $G$ such that $e\in E^{\prime}\cup E^{\prime\prime}$ or $e$ consists of
two semiedges one of which is included in $S^{\prime}$ and the other in
$S^{\prime\prime},$ and finally let $S^{\prime\prime\prime}$ be a set of
semiedges $s$ such that $s\in S^{\prime}\cup S^{\prime}$ and $s$ is not a half
of an edge from $E^{\prime\prime\prime}.$ A \emph{union} of multipoles
$M^{\prime}$ and $M^{\prime\prime}$ is a multipole $M^{\prime}\cup
M^{\prime\prime}=(V^{\prime\prime\prime},E^{\prime\prime\prime},S^{\prime
\prime\prime}).$

\begin{remark}
Let $M$ be a multipole, and $M^{\prime}$ and $M^{\prime\prime}$ its two
submultipoles as in Figure \ref{Fig_submultipole}. Notice that both
$M^{\prime}$ and $M^{\prime\prime}$ are induced submultipoles of $M,$ where
$M^{\prime}$ is induced by $\{u_{1},u_{2},v_{2}\}$ and $M^{\prime\prime}$ by
$\{u_{3},u_{4},v_{4}\}.$ The union $M^{\prime}\cup M^{\prime\prime}$ is a
multipole $(V^{\prime\prime\prime},E^{\prime\prime\prime},S^{\prime
\prime\prime})$ where
\begin{align*}
V^{\prime\prime\prime}  &  =\{u_{1},u_{2},u_{3},u_{4},v_{2},v_{4}\},\\
E^{\prime\prime\prime}  &  =\{u_{1}u_{2},u_{2}u_{3},u_{3}u_{4},u_{2}%
v_{2},u_{4}v_{4}\},\\
S^{\prime\prime\prime}  &  =\{s_{1},s_{1}^{\prime},s_{2}^{\prime}%
,s_{2}^{\prime\prime},s_{3}^{\prime},s_{4},s_{4}^{\prime},s_{4}^{\prime\prime
}\}.
\end{align*}
The semiedges $s_{2}$ and $s_{3}$ which belong to $M^{\prime}$ and
$M^{\prime\prime},$ respectively, do not belong to the set of semiedges of
$M^{\prime}\cup M^{\prime\prime},$ they form the edge $u_{2}u_{3}$ in
$E^{\prime\prime\prime}$ instead.
\end{remark}

\begin{figure}[h]
\begin{center}
\includegraphics[scale=0.6]{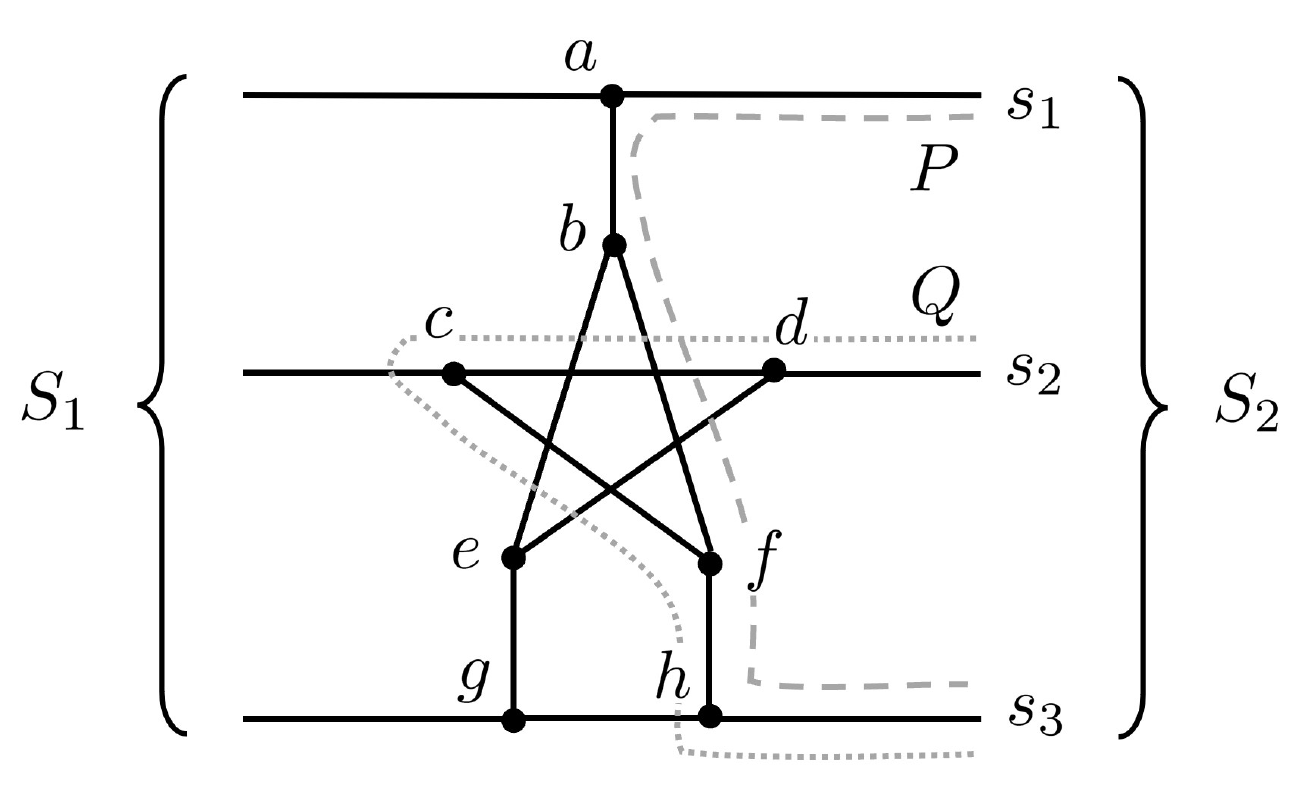}
\end{center}
\caption{Superedge $B$ with two connectors $S_{1}$ and $S_{2}.$ Paths $P$
(dashed) and $Q$ (dotted) starting and ending with semiedges from $S_{2}$ are
also denoted.}%
\label{Fig_superEdge}%
\end{figure}

Let $M$ be a multipole and $M^{\prime}$ and $M^{\prime\prime}$ two
submultipoles of $M.$ Let $\sigma^{\prime}$ and $\sigma^{\prime\prime}$ be a
normal coloring of $M^{\prime}$ and $M^{\prime\prime}$, respectively. We say
that $\sigma^{\prime}$ and $\sigma^{\prime\prime}$ are \emph{compatible}, if
there exists a normal coloring $\sigma$ of $M^{\prime}\cup M^{\prime\prime}$
such that the restriction of $\sigma$ to $E^{\prime}\cup S^{\prime}$ is
$\sigma^{\prime}$ and the restriction to $E^{\prime\prime}\cup S^{\prime
\prime}$ is $\sigma^{\prime\prime}$. If $\sigma^{\prime}$ and $\sigma
^{\prime\prime}$ are compatible, we also say that $\sigma^{\prime}$ is
$\sigma^{\prime\prime}$\emph{-compatible} and vice versa. Notice that there
may exist an edge $e$ in $M^{\prime}\cup M^{\prime\prime}$ such that the half
of $e$ is semiedge in $M^{\prime}$ or $M^{\prime\prime},$ say in $M^{\prime}$,
then the restriction of $\sigma$ to $M^{\prime}$ means that $\sigma
(e)=\sigma^{\prime}(s)$.

Throughout the paper we will consider supervertices $A$ and $A^{\prime}$
illustrated by Figure \ref{Fig_superVert}, and a supervertex $B$ defined as
$(P_{10})_{u_{1},u_{2}}$ where $u_{1}$ and $u_{2}$ is any pair of non-adjacent
vertices of $P_{10}$ and illustrated by Figure \ref{Fig_superEdge}. Assuming
that vertices and semiedges of $B$ are labeled as in Figure
\ref{Fig_superEdge}, we define paths $P$ and $Q$ in $B$ by $P=s_{1}%
,ab,bf,fh,s_{3}$ and $Q=s_{2},dc,cf,fh,s_{3}$. These two paths are also shown
in Figure \ref{Fig_superEdge}. Notice that both $P$ and $Q$ are paths starting
and ending with a semiedge, so they could be Kempe chains for some normal
colorings of $B$.

\begin{figure}[h]
\begin{center}
\includegraphics[scale=0.6]{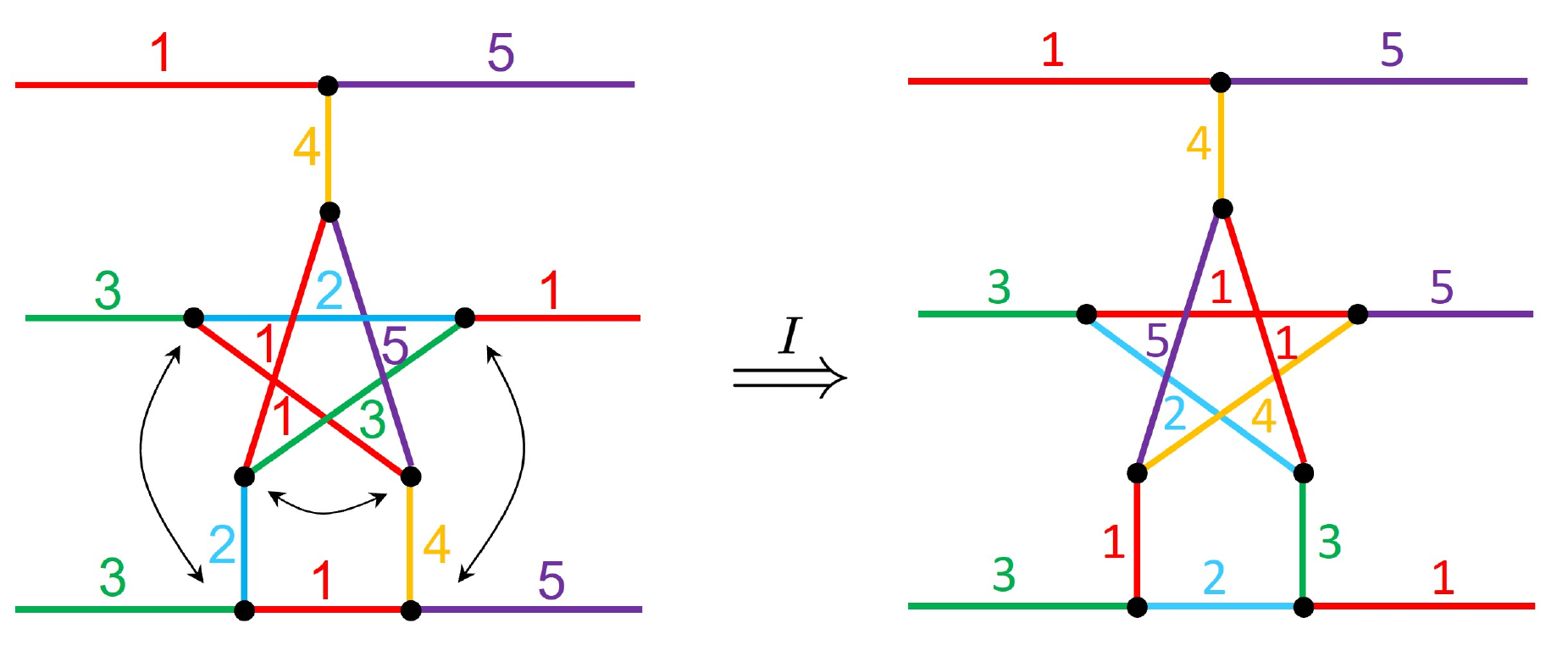}
\end{center}
\caption{The figure shows a normal $5$-coloring $\sigma$ of the superedge $B$
on the left, and the coloring $I(\sigma)$ on the right which is also a normal
$5$-coloring of $B.$}%
\label{Fig_isomorphismB}%
\end{figure}

Notice that in the superedge $B$ the horizontal "level" containing vertices
$c$ and $d$ can be swapped with the horizontal "level" containing vertices $g$
and $h,$ and we will again obtain a superedge $B$ provided that vertices $e$
and $f$ are swapped too. To be more precise, we define an isomorphism $I$ of
$B$ by the following permutation of vertices $(c$ $g)(d$ $h)(e$ $f)$
considered as a product of three transpositions. Note that $I$ is involution,
i.e. $I=I^{-1}.$ If $\sigma$ is a normal coloring of $B$ and we assume that in
the isomorphism $I$ the colors of edges are mapped together with edges, this
yields another normal coloring of $B.$ Formally, for a given normal
$5$-coloring $\sigma$ of $B$ we define an edge-coloring $I(\sigma)$ of $B$ by
$I(\sigma)(x)=\sigma(I^{-1}(x))$ for any $x\in E(B)\cup S(B).$ The following
observation on $I(\sigma)$ is illustrated by Figure \ref{Fig_isomorphismB}.

\begin{remark}
If $\sigma$ is a normal $5$-coloring of $B,$ then $I(\sigma)$ is also a normal
$5$-coloring of $B$.
\end{remark}

\begin{figure}[h]
\begin{center}
\includegraphics[scale=0.6]{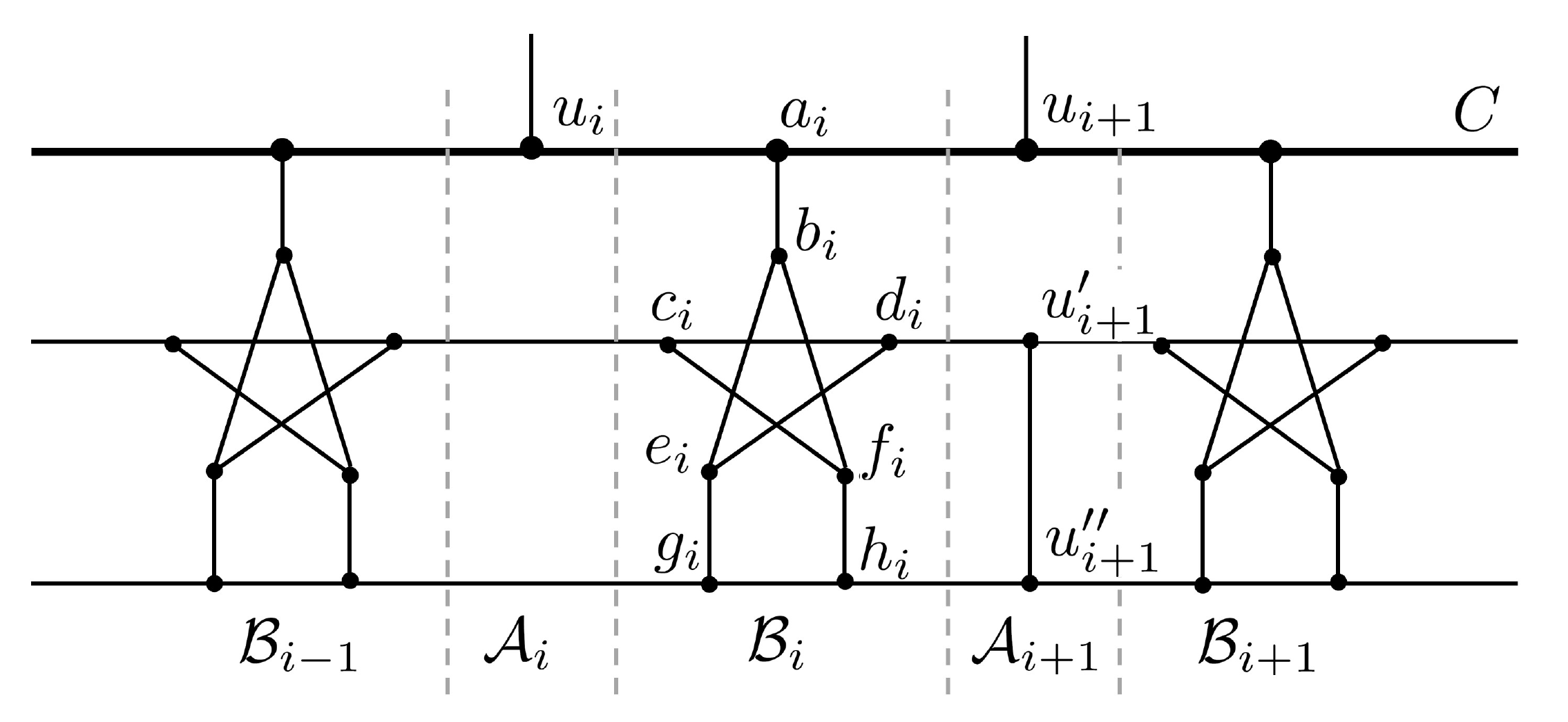}
\end{center}
\caption{Superposition of vertices and edges of a cycle $C$ in $G$ by
supervertices $A$ or $A^{\prime}$ and superedge $B.$}%
\label{Fig_Kochol}%
\end{figure}

Let $G$ be a snark, and $C=u_{0}u_{1}\cdots u_{g-1}u_{0}$ a cycle of length
$g$ in $G.$ Denote by $v_{i}$ the neighbor of $u_{i}$ distinct from $u_{i-1}$
and $u_{i+1},$ for $i=0,\ldots,g-1.$ Notice that it can happen $v_{i}=u_{j}$
for $j\not \in \{i+1,i,i-1\}$ and also $v_{i}=v_{j}$ for $i\not =j.$ Let
$G_{C}(\mathcal{A},\mathcal{B})$ be a superposition of a snark $G$ such that
every vertex from $C$ is superposed by $A$ or $A^{\prime},$ every edge from
$C$ is superposed by $B,$ and all other edges and vertices of $G$ are
superposed by themselves. Moreover, the semiedges of $A$ (resp. $A^{\prime}$)
are identified with semiedges of $B$ as in Figure \ref{Fig_Kochol}. We will
use abbreviated notation $\mathcal{A}_{i}$ for $\mathcal{A}(u_{i})$ and
$\mathcal{B}_{i}$ for $\mathcal{B}(u_{i}u_{i+1}).$ The family of all such
superpositions of $G$ is denoted by $\mathcal{G}_{C}(\mathcal{A}%
,\mathcal{B}).$

Notice that the semiedges of the two $3$-connectors of $\mathcal{A}_{i}$ can
be identified with semiedges of a $3$-connector in $\mathcal{B}_{i-1}$ and
$\mathcal{B}_{i}$ in several ways, and the one from Figure \ref{Fig_Kochol}
which we consider is only one of them. Namely, we may assume that the
connection of $\mathcal{A}_{i}$ with $\mathcal{B}_{i-1}$ and $\mathcal{B}_{i}$
arises so that first each of the three semiedges of $\mathcal{B}_{i-1}$ is
identified with a semiedge from $\mathcal{B}_{i},$ and then one of the three
arising edges is subdivided by $u_{i}.$ Since the identification of semiedges
from $\mathcal{B}_{i-1}$ and $\mathcal{B}_{i}$ can be done in $6$ different
ways, and then the edge to be subdivided by $u_{i}$ can be chosen in $3$
different ways, there are $18$ different ways of connecting supervertex
$\mathcal{A}_{i}$ with superedges $\mathcal{B}_{i-1}$ and $\mathcal{B}_{i}.$

\begin{figure}[h]
\begin{center}
\includegraphics[scale=0.7]{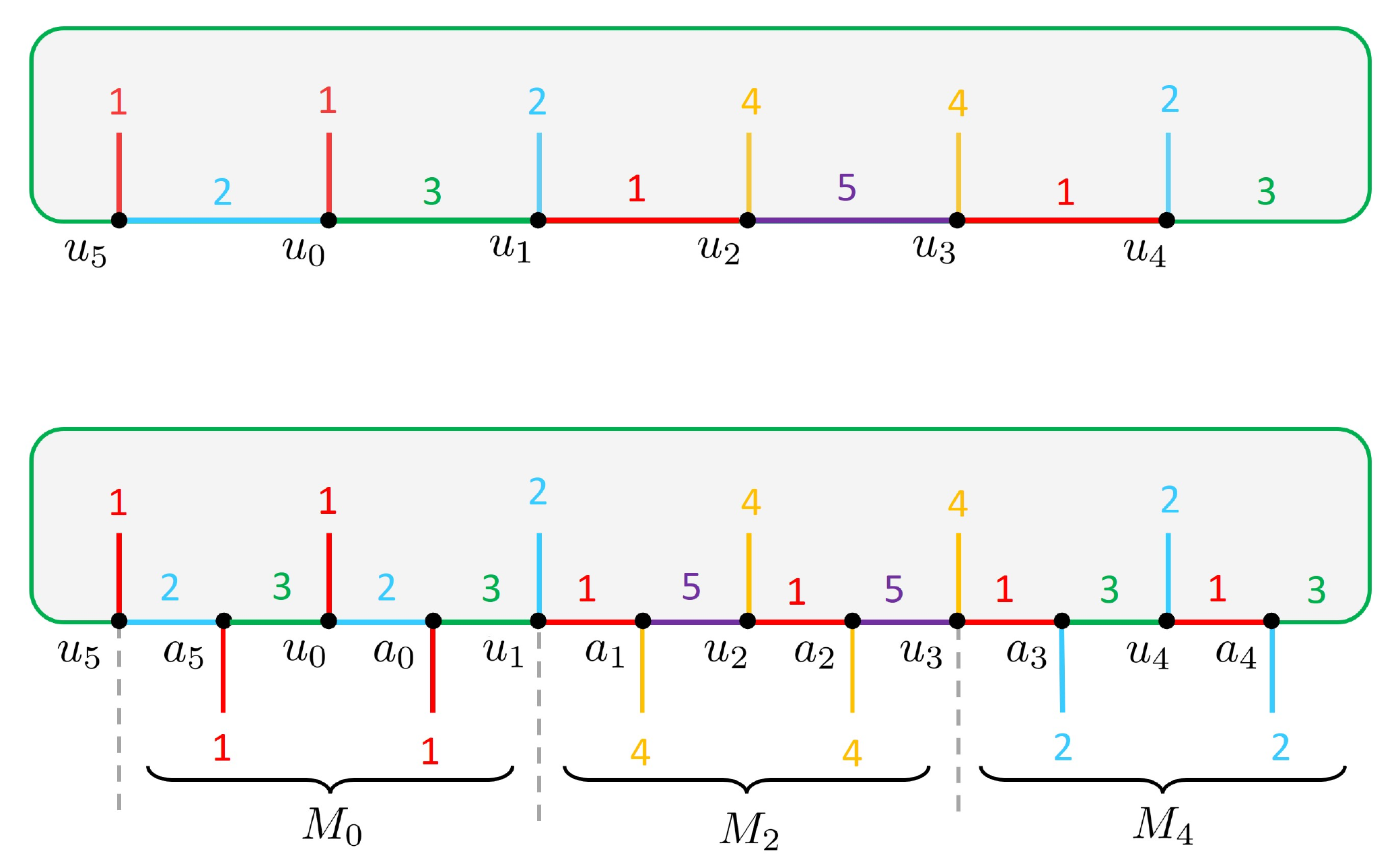}
\end{center}
\caption{A normal coloring $\sigma$ of edges incident to a $6$-cycle $C$ in
$G,$ and a normal coloring $\tilde{\sigma}_{\mathrm{ext}}$ of the multipole
$M_{\mathrm{ext}}$ such that the restriction of $\tilde{\sigma}_{\mathrm{ext}%
}$ to $E(G)\backslash E(C)$ is equal to $\sigma.$}%
\label{Fig_rotation}%
\end{figure}

\begin{remark}
Since $\mathcal{A}_{i}\in\{A,A^{\prime}\}$, $\mathcal{B}_{i}=B$ and a
supervertex $\mathcal{A}_{i}$ is connected to $\mathcal{B}_{i-1}$ and
$\mathcal{B}_{i}$ in one particular way out of 18 possible, it holds that
$\left\vert \mathcal{G}_{C}(\mathcal{A},\mathcal{B})\right\vert =2^{g}.$
\end{remark}

\noindent Note that if all $18$ ways of connecting $\mathcal{A}_{i}$ to
$\mathcal{B}_{i-1}$ and $\mathcal{B}_{i}$ were considered, then $\mathcal{G}%
_{C}(\mathcal{A},\mathcal{B})$ would contain $36^{g}$ superpositions of $G$
(considering them as labeled graphs).

\begin{theorem}
\label{Tm_main}Let $G$ be a snark which has a normal $5$-coloring $\sigma$ and
let $C$ be an even cycle in $G$. Then every $\tilde{G}\in\mathcal{G}%
_{C}(\mathcal{A},\mathcal{B})$ has a normal $5$-coloring $\tilde{\sigma}$ such
that $\tilde{\sigma}(e)=\sigma(e)$ for every $e\in E(G)\backslash E(C)$ and
$\tilde{\sigma}$ has at least $18$ poor edges.
\end{theorem}

\begin{proof}
Let $\tilde{G}_{A}\in\mathcal{G}_{C}(\mathcal{A},\mathcal{B})$ be a
superposition of $G$ such that $\mathcal{A}_{i}=A$ for every vertex $u_{i}\in
V(C).$ We will first define a normal $5$-coloring $\tilde{\sigma}_{A}$ of
$\tilde{G}_{A},$ and then we will show that a slight modification of
$\tilde{\sigma}_{A}$ suffices to obtain a normal $5$-coloring of any
$\tilde{G}\in\mathcal{G}_{C}(\mathcal{A},\mathcal{B}).$ Assume that vertices
within $\mathcal{A}_{i}$ and $\mathcal{B}_{i}$ of $\tilde{G}_{A}$ are denoted
as in Figure \ref{Fig_Kochol}.

Denote by $M_{\mathrm{{int}}}$ a submultipole of $\tilde{G}_{A}$ induced by
$V(G)\backslash V(C)\subseteq V(\tilde{G}_{A}).$ Also, let $M_{\mathrm{ext}%
}=(V_{\mathrm{ext}},E_{\mathrm{ext}},S_{\mathrm{ext}})$ be a submultipole of
$\tilde{G}_{A}$ induced by $V(G)\cup\{a_{i}:i=0,\ldots,g-1\}\subseteq
V(\tilde{G}_{A}).$ We define an edge-coloring $\tilde{\sigma}_{\mathrm{ext}}$
of $M_{\mathrm{ext}}$ as follows
\[
\tilde{\sigma}_{\mathrm{ext}}(x)=\left\{
\begin{array}
[c]{ll}%
\sigma(x) & \text{if }x\in E(G)\backslash E(C)\subseteq E_{\mathrm{ext}},\\
\sigma(u_{2i-1}u_{2i}) & \text{if }x\in\{u_{2i-1}a_{2i-1},u_{2i}%
a_{2i}\}\subseteq E_{\mathrm{ext}},\\
\sigma(u_{2i}u_{2i+1}) & \text{if }x\in\{a_{2i-1}u_{2i},a_{2i}u_{2i+1}%
\}\subseteq E_{\mathrm{ext}},\\
\sigma(u_{2i}v_{2i}) & \text{if }x\in\{a_{2i-1}b_{2i-1},a_{2i}b_{2i}%
\}\subseteq S_{\mathrm{ext}}.
\end{array}
\right.
\]
The coloring $\tilde{\sigma}_{\mathrm{ext}}$ of the multipole $M_{\mathrm{ext}%
}$ is illustrated by Figure \ref{Fig_rotation}. Since $M_{\mathrm{int}}$ is a
submultipole of $M_{\mathrm{ext}}$ and $\left.  \tilde{\sigma}_{\mathrm{ext}%
}\right\vert _{M_{\mathrm{int}}}=\left.  \sigma\right\vert _{M_{\mathrm{int}}%
},$ we say that $\tilde{\sigma}_{\mathrm{ext}}$ is the extension of $\left.
\sigma\right\vert _{M_{\mathrm{int}}}$ to $M_{\mathrm{ext}}.$

\bigskip\noindent\textbf{Claim A.} \emph{It holds that }$\tilde{\sigma
}_{\mathrm{ext}}$\emph{ is a normal }$5$\emph{-coloring of }$M_{\mathrm{ext}%
}.$

\medskip\noindent Let us establish that $\tilde{\sigma}_{\mathrm{ext}}$ is a
proper $5$-coloring. Notice that for any vertex $v\in V(G)\subseteq
V(M_{\mathrm{ext}})$ it holds that $\tilde{\sigma}_{\mathrm{ext}}%
(v)=\sigma(v).$ Also, notice that $V(M_{\mathrm{ext}})\backslash
V(G)=\{a_{i}:i=0,\ldots,g-1\},$ and it holds that $\tilde{\sigma
}_{\mathrm{ext}}(a_{2i-1})=\tilde{\sigma}_{\mathrm{ext}}(a_{2i})=\sigma
(u_{2i})$ for every $0\leq i\leq g/2.$ Thus, $\tilde{\sigma}_{\mathrm{ext}}$
is proper since $\sigma$ is proper, and it uses the same colors as $\sigma$.
It remains to prove that $\tilde{\sigma}_{\mathrm{ext}}$ is normal. If $e=uv$
is an edge of $M_{\mathrm{ext}}$ which is not incident to vertices $u_{i}$ of
the cycle $C,$ notice that $\sigma(e)=\tilde{\sigma}_{\mathrm{ext}}(e),$
$\sigma(u)=\tilde{\sigma}_{\mathrm{ext}}(u)$ and $\sigma(v)=\tilde{\sigma
}_{\mathrm{ext}}(v).$ Hence, since $e$ is normal by $\sigma$ in $G$ it follows
that $e$ is normal by $\tilde{\sigma}_{\mathrm{ext}}$ in $M_{\mathrm{ext}}.$

If $e=u_{2i}a_{2i}$ or $a_{2i-1}u_{2i},$ then $\sigma_{\mathrm{ext}}%
(a_{2i-1})=\sigma_{\mathrm{ext}}(u_{2i})=\sigma_{\mathrm{ext}}(a_{2i}%
)=\sigma(u_{2i})$ implies that $e$ is poor by $\tilde{\sigma}_{\mathrm{ext}}$
in $M_{\mathrm{ext}}.$

If $e=a_{2i}u_{2i+1}$ (resp. $e=u_{2i-1}a_{2i-1}$), notice that $\sigma
(e)=\sigma(u_{2i}u_{2i+1})$ (resp. $\sigma(e)=\sigma(u_{2i-1}u_{2i})$) and as
for end-vertices of $e$ it holds that $\tilde{\sigma}_{\mathrm{ext}}%
(a_{2i})=\sigma(u_{2i})$ and $\tilde{\sigma}_{\mathrm{ext}}(u_{2i+1}%
)=\sigma(u_{2i+1})$ (resp. $\tilde{\sigma}_{\mathrm{ext}}(a_{2i-1}%
)=\sigma(u_{2i})$ and $\tilde{\sigma}_{\mathrm{ext}}(u_{2i-1})=\sigma
(u_{2i-1})$), thus $e$ is normal by $\tilde{\sigma}_{\mathrm{ext}}$ in
$M_{\mathrm{ext}}$ because $u_{2i}u_{2i+1}$ (resp. $u_{2i-1}u_{2i}$) is normal
by $\sigma$ in $G,$ which establishes the claim.

\medskip

Now, for a semiedge $s$ incident to a vertex $v$ and a normal edge-coloring
$\sigma$ of a cubic multipole $M$, the \emph{complete color} $\sigma
^{\mathrm{c}}(s)$ of $s$ is defined by $\sigma^{\mathrm{c}}(s)=(i,\{j,k\})$
where $i=\sigma(s)$ and $\{j,k\}$ is the set of the two colors of the
remaining two (semi)edges incident to $v$. Let $M_{i}$ denote the submultipole
of $\tilde{G}_{A}$ induced by $V(\mathcal{B}_{i-1})\cup V(\mathcal{A}_{i})\cup
V(\mathcal{B}_{i}),$ see Figure \ref{Fig_Kochol}. Let $\tilde{\sigma}_{i}$ be
a normal $5$-coloring of $M_{i}$ which is $\tilde{\sigma}_{\mathrm{ext}}%
$-compatible. Denote by $s_{1},$ $s_{2}$ and $s_{3}$ the semiedges of $M_{i}$
incident to vertices $a_{i},$ $d_{i}$ and $h_{i}$, respectively. We say
$\tilde{\sigma}_{i}$ is \emph{right-side monochromatic} if the complete color
of all three semiedges $s_{i}$ is the same, i.e. $\tilde{\sigma}%
_{i}^{\mathrm{c}}(s_{1})=\tilde{\sigma}_{i}^{\mathrm{c}}(s_{2})=\tilde{\sigma
}_{i}^{\mathrm{c}}(s_{3}).$ Further, we say that $\tilde{\sigma}_{i}$ is
\emph{left-side compatible} if it is compatible with a right-side
monochromatic $\tilde{\sigma}_{\mathrm{ext}}$-compatible normal $5$-coloring
of $M_{i-2}.$

Notice the following, if every $M_{2i}$ is colored by a normal $5$-coloring
$\tilde{\sigma}_{2i}$ which is $\tilde{\sigma}_{\mathrm{ext}}$-compatible,
then we obtain a normal $5$-coloring of $M_{\mathrm{ext}}\cup M_{2i}$.
Further, if $\tilde{\sigma}_{2i}$ is also right-side monochromatic and
left-side compatible for every $i=0,\ldots,g/2,$ then every $\tilde{\sigma
}_{2i}$ is compatible with $\tilde{\sigma}_{2i-2}$. Thus combining
$\tilde{\sigma}_{\mathrm{ext}}$ with $\tilde{\sigma}_{2i}$ for all
$i=0,\ldots,g/2$ would yield a normal $5$-coloring of $\tilde{G}.$ Hence, we
first have to establish that every $M_{2i}$ indeed has a normal $5$-coloring,
which is $\tilde{\sigma}_{\mathrm{ext}}$-compatible, right-side monochromatic
and left-side compatible.

\begin{figure}[ph]
\begin{center}%
\begin{tabular}
[t]{llll}%
a) & \raisebox{-0.9\height}{\includegraphics[scale=0.45]{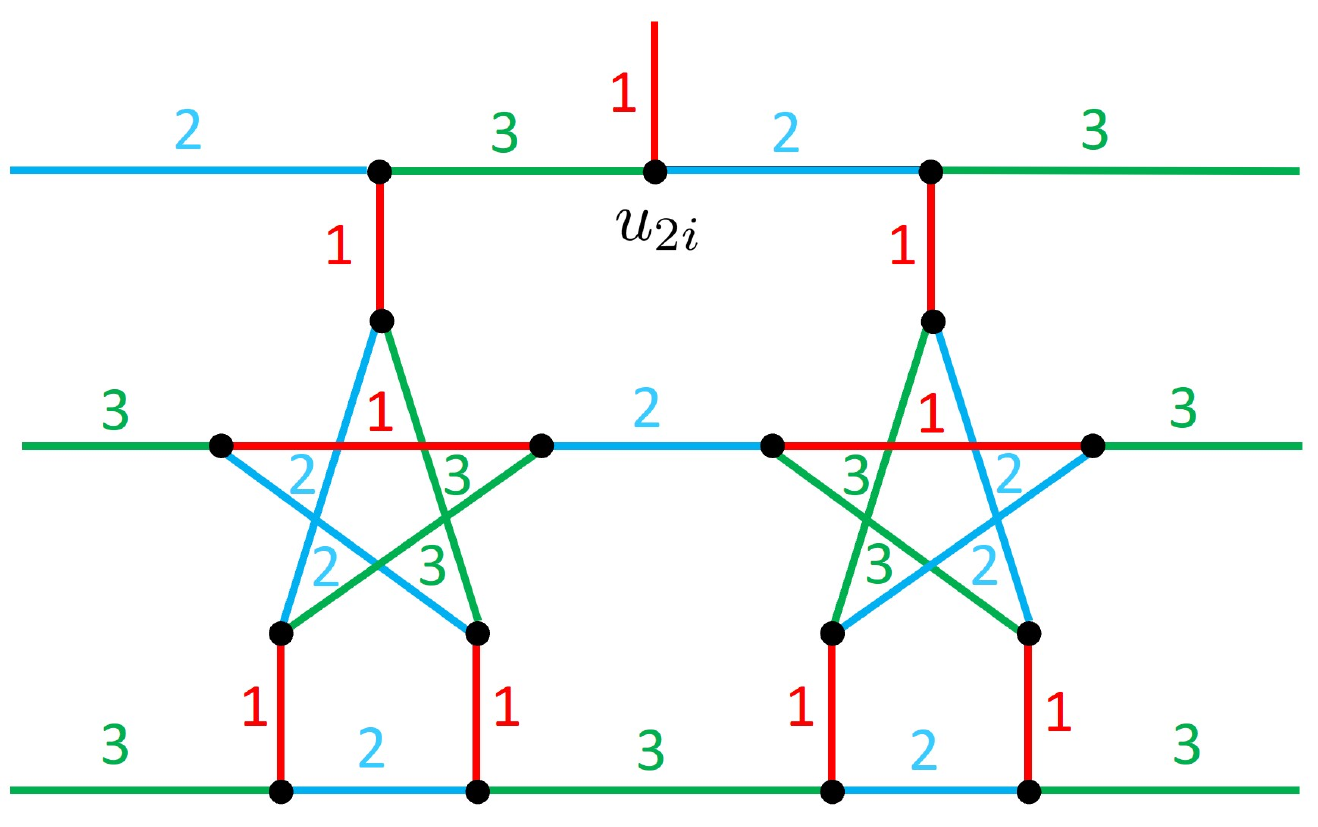}} & b) &
\raisebox{-0.9\height}{\includegraphics[scale=0.45]{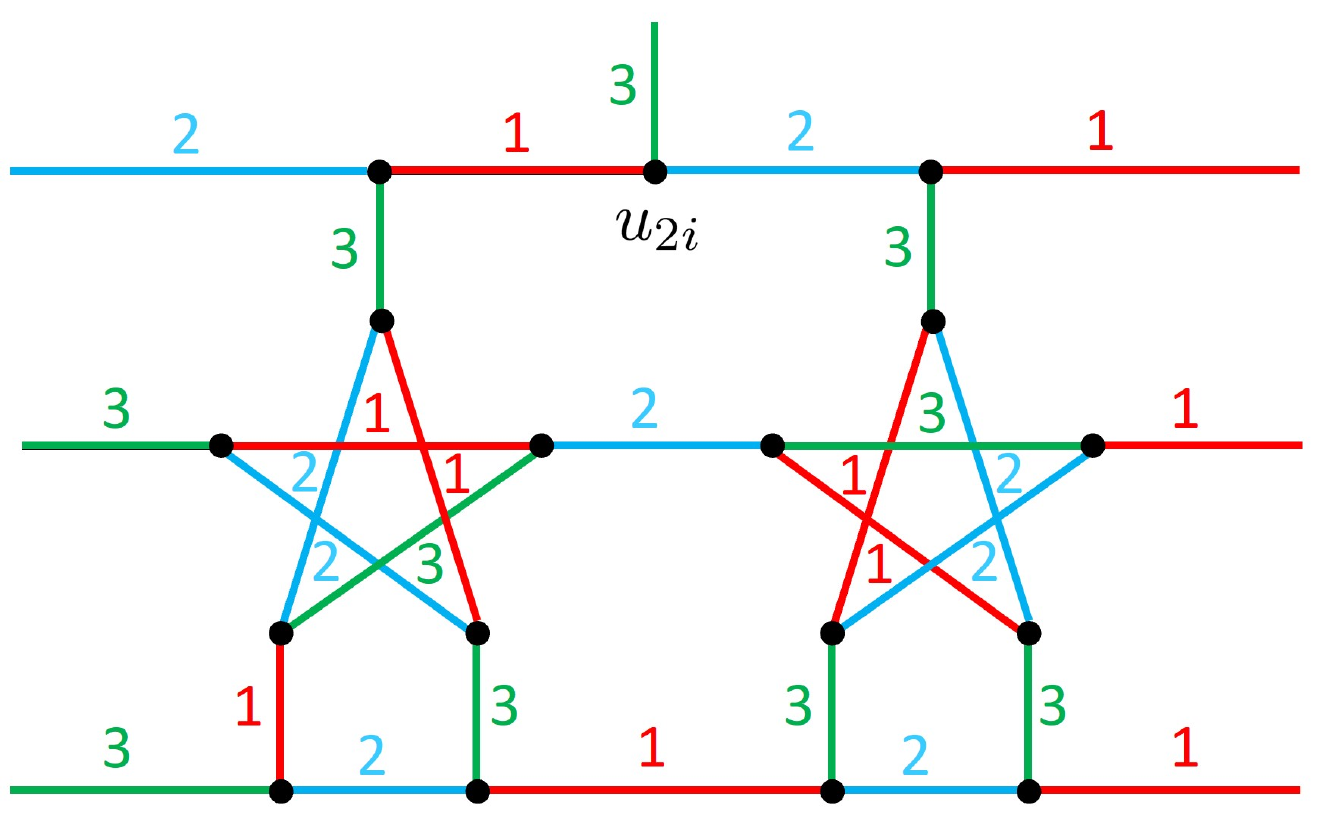}}\\
c) & \raisebox{-0.9\height}{\includegraphics[scale=0.45]{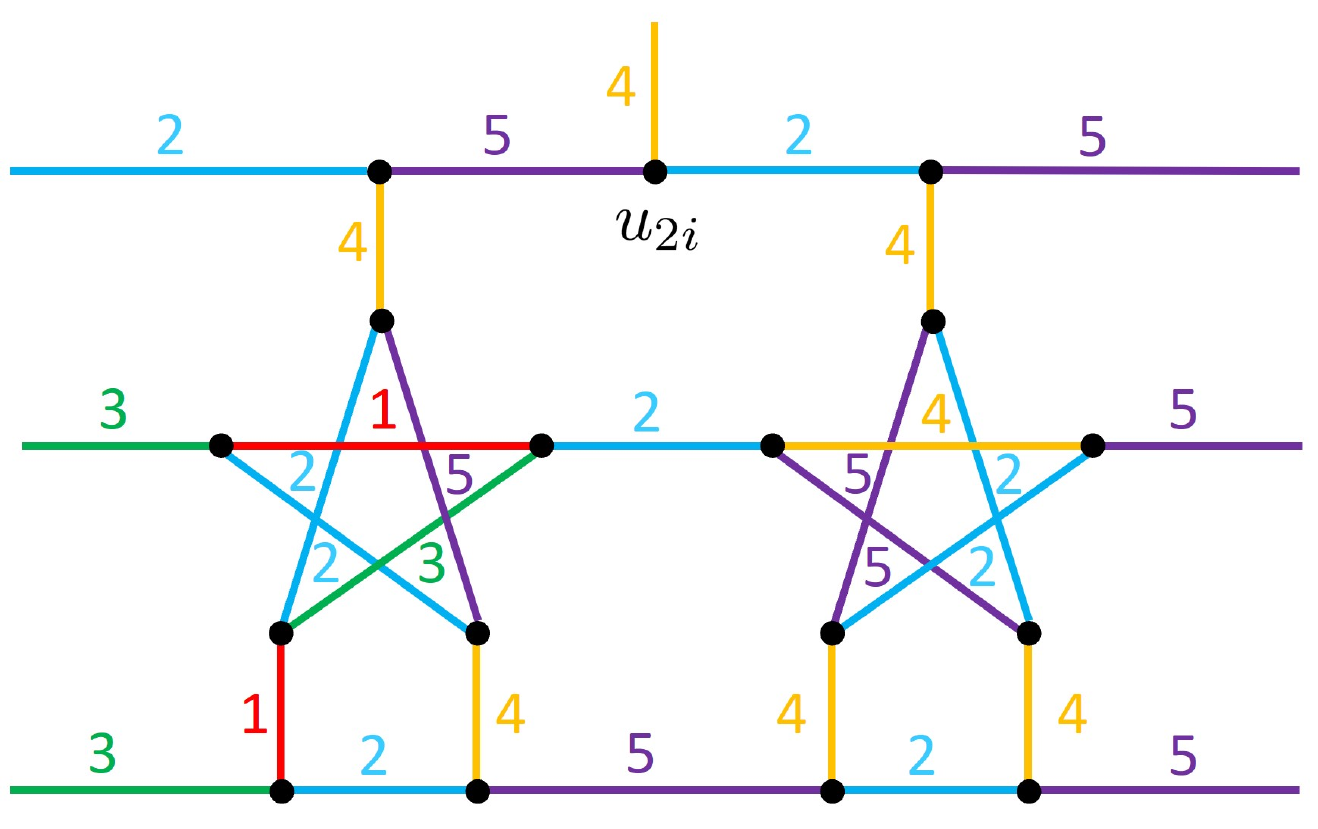}} & d) &
\raisebox{-0.9\height}{\includegraphics[scale=0.45]{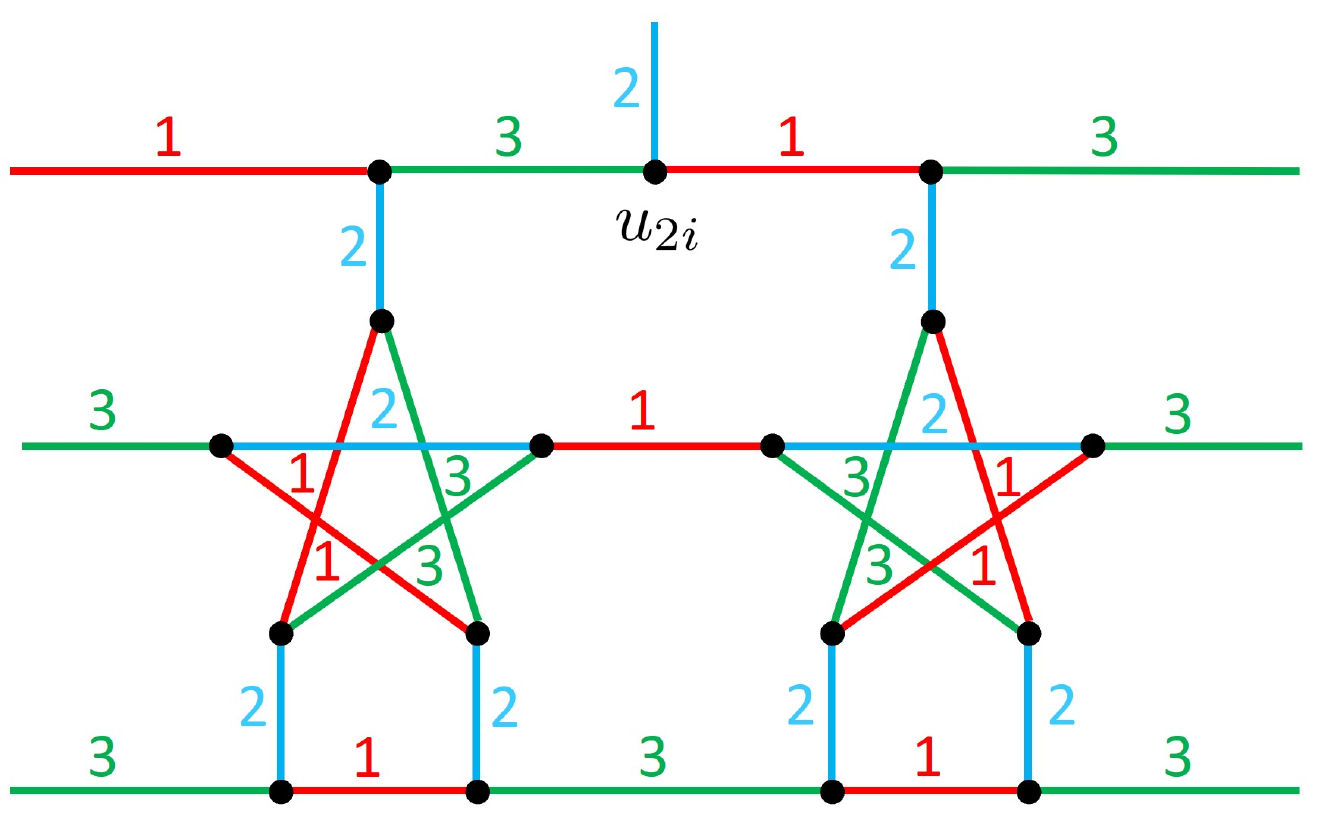}}\\
e) & \raisebox{-0.9\height}{\includegraphics[scale=0.45]{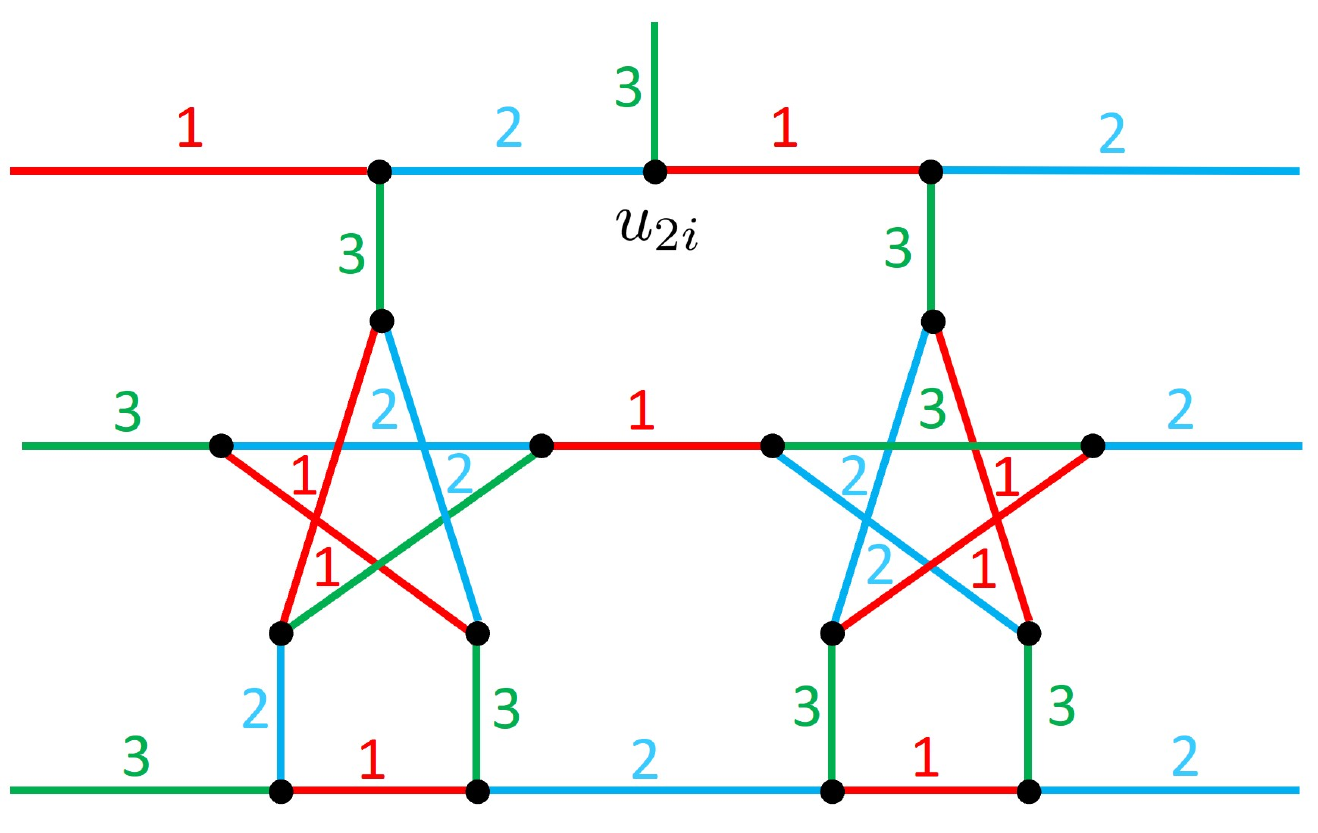}} & f) &
\raisebox{-0.9\height}{\includegraphics[scale=0.45]{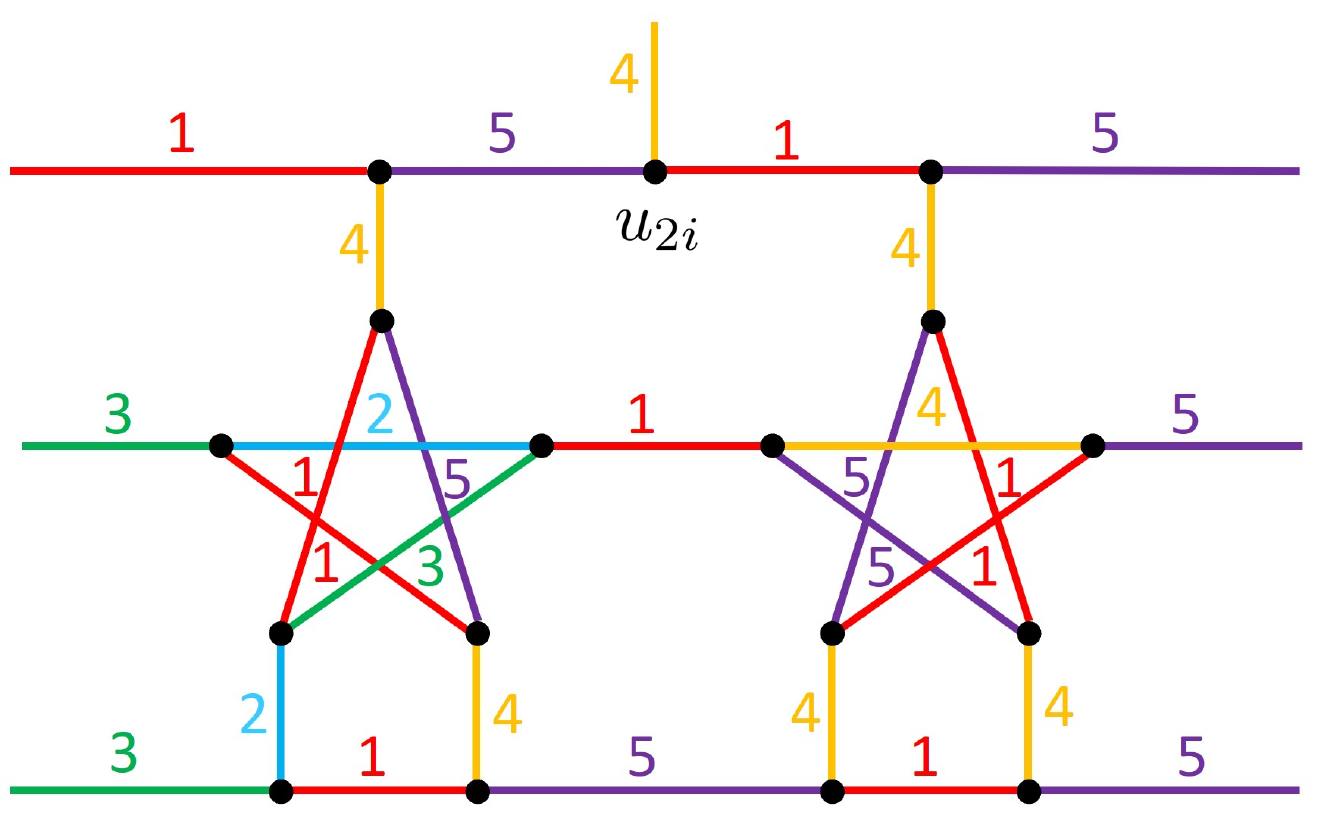}}\\
g) & \raisebox{-0.9\height}{\includegraphics[scale=0.45]{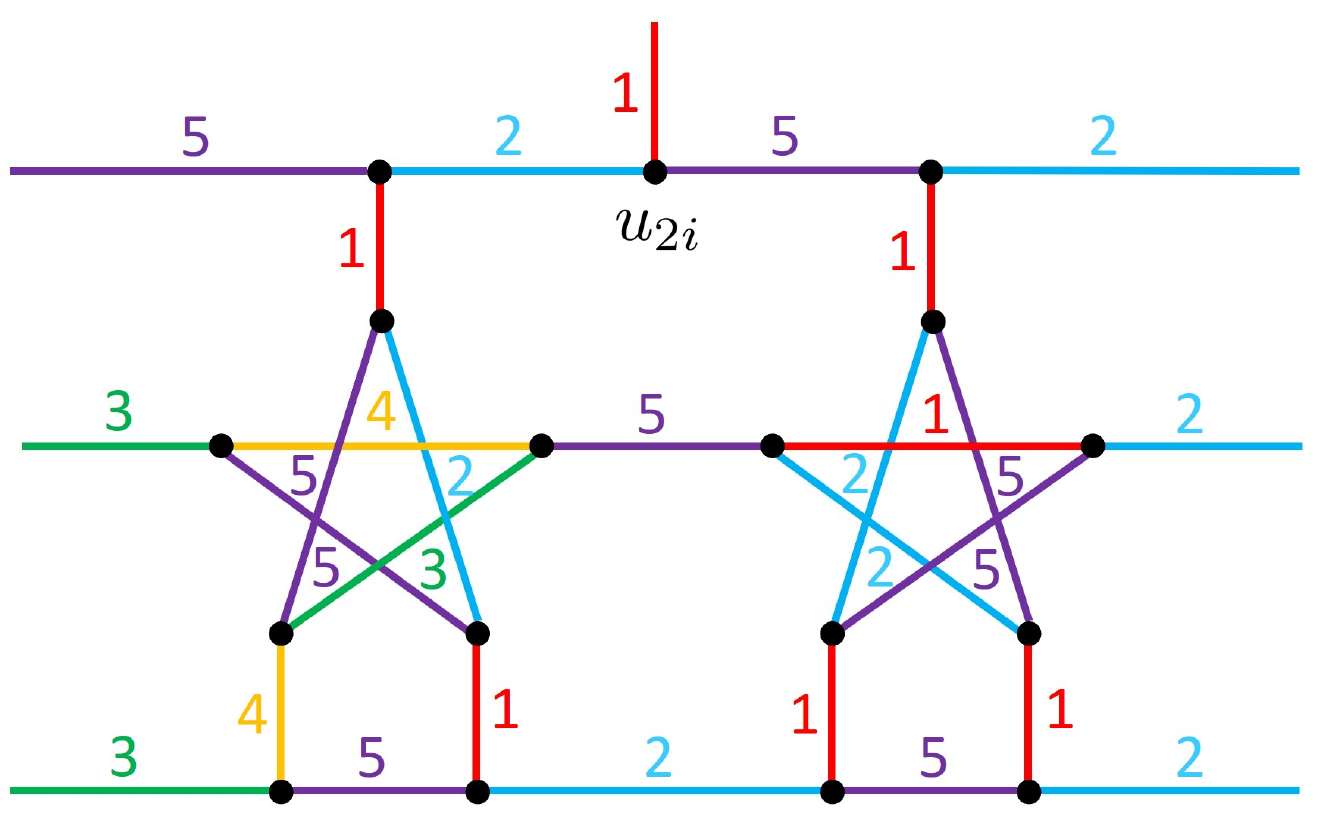}} & h) &
\raisebox{-0.9\height}{\includegraphics[scale=0.45]{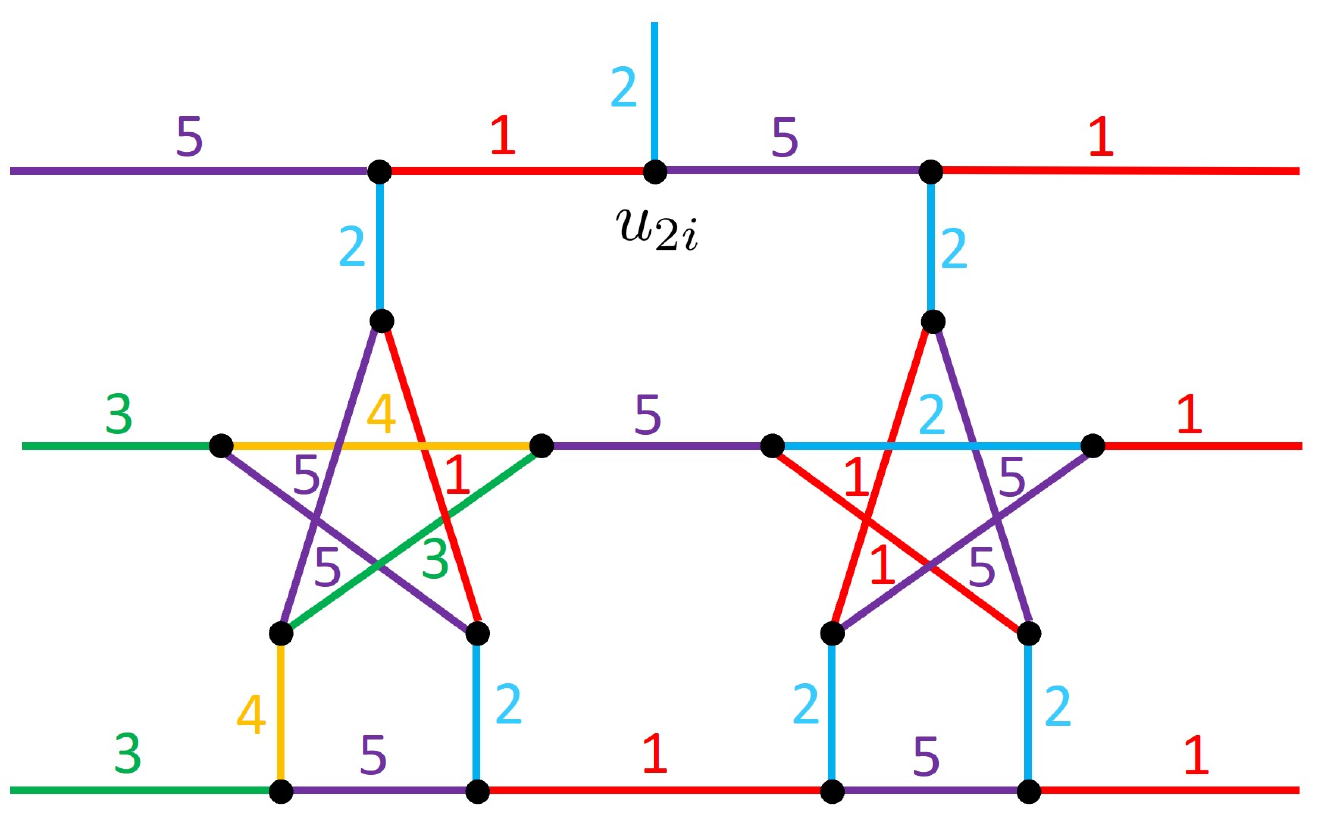}}\\
i) & \raisebox{-0.9\height}{\includegraphics[scale=0.45]{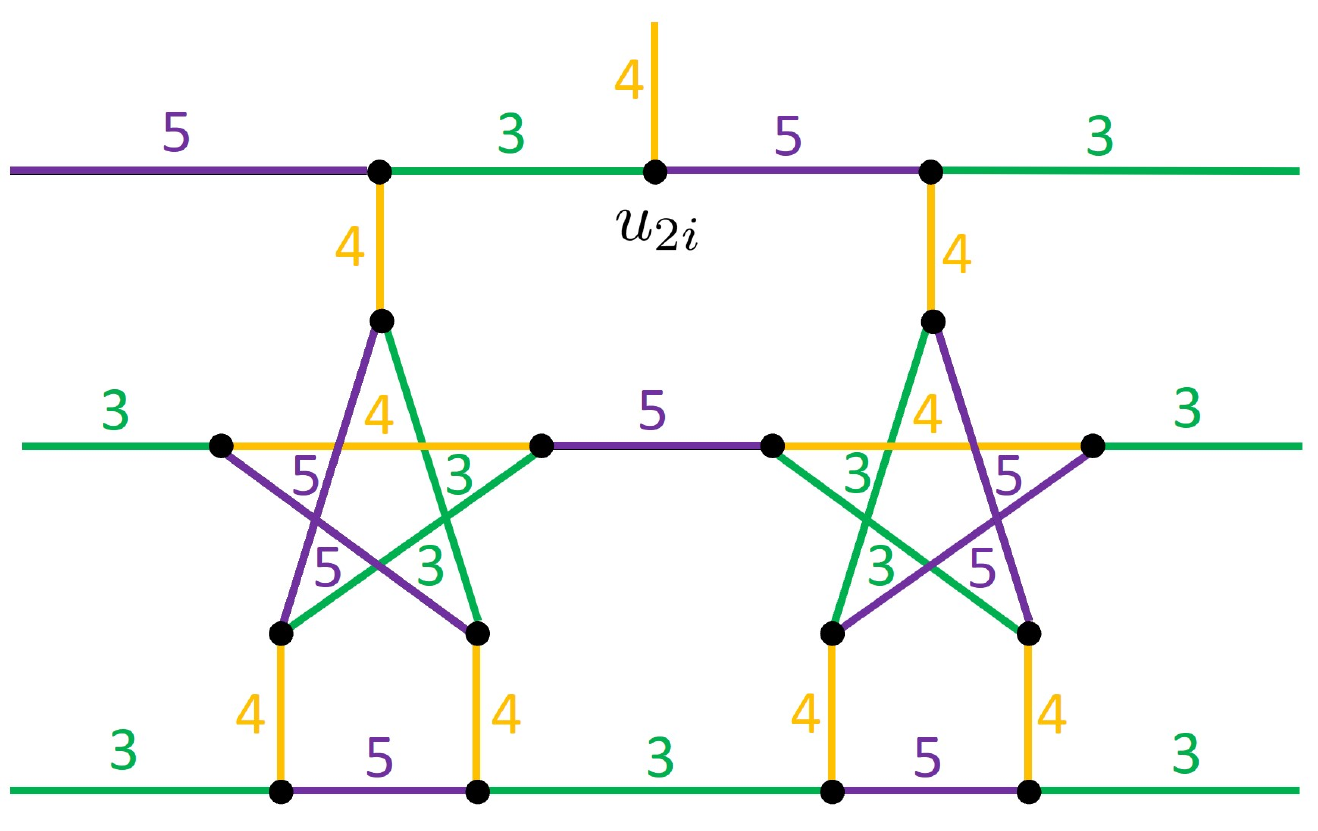}} & j) &
\raisebox{-0.9\height}{\includegraphics[scale=0.45]{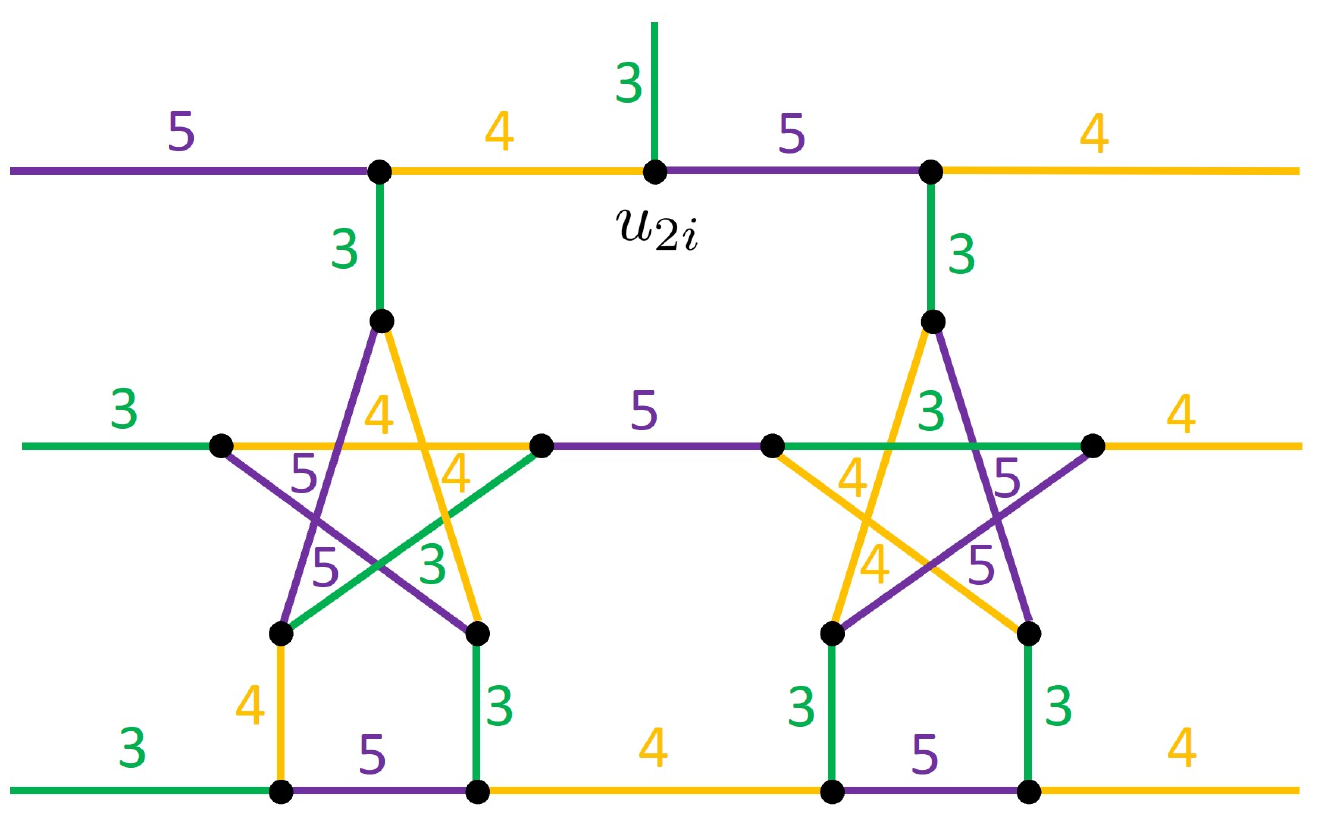}}
\end{tabular}
\end{center}
\caption{The coloring $\tilde{\sigma}_{2i}$ of $M_{2i}$ for $\sigma_{t}%
(u_{2i})$ equal respectively to: a) $(1,2,3),$ b) $(3,2,1),$ c) $(4,2,5),$ d)
$(2,1,3),$ e) $(3,1,2),$ f) $(4,1,5),$ g) $(1,5,2),$ h) $(2,5,1),$ i)
$(4,5,3),$ j) $(3,5,4).$}%
\label{Fig_tenDoubleStars}%
\end{figure}

\bigskip\noindent\textbf{Claim B.} \emph{For every }$i=0,\ldots,g/2$\emph{,
there exists a normal }$5$-coloring\emph{ }$\tilde{\sigma}_{2i}$\emph{ of
}$M_{2i}$\emph{ with all the following properties:}

\begin{enumerate}
\item[i)] $\tilde{\sigma}_{2i}$\emph{ is }$\tilde{\sigma}_{\mathrm{ext}}%
$\emph{-compatible;}

\item[ii)] $\tilde{\sigma}_{2i}$\emph{ is right-side monochromatic in colors
}$(\sigma(u_{2i}u_{2i+1}),\{\sigma(u_{2i}v_{2i}),\sigma(u_{2i-1}u_{2i}%
)\})$\emph{ of the corresponding edges of cycle }$C$\emph{ in }$G;$

\item[iii)] $\tilde{\sigma}_{2i}$\emph{ is left-side compatible.}
\end{enumerate}

\medskip\noindent For a vertex $u_{2i}\in V(C),$ let us denote $\sigma
_{t}(u_{2i})=(\sigma(u_{2i}v_{2i}),\sigma(u_{2i-1}u_{2i}),\sigma
(u_{2i}u_{2i+1})).$ We may assume that $\sigma_{t}(u_{2i-2})=(1,2,3),$ and
since $\sigma$ is a normal $5$-coloring this further implies $\sigma
_{t}(u_{2i})\in T$ where (up to permutation of colors $4$ and $5$) it holds
that%
\begin{align*}
T  &  =\{(1,2,3),(3,2,1),(4,2,5),(2,1,3),(3,1,2),(4,1,5),\\
&  (1,5,2),(2,5,1),(4,5,3),(3,5,4)\}.
\end{align*}
If $\sigma_{t}(u_{2i})=(1,2,3),$ then $\tilde{\sigma}_{2i}$ of $M_{2i}$ is
defined as in Figure \ref{Fig_tenDoubleStars}.a). Denote by $\kappa$ the
restriction of $\tilde{\sigma}_{2i}$ to $\mathcal{B}_{2i}$ and by $\tau$ the
restriction of $\tilde{\sigma}_{2i}$ to $\mathcal{B}_{2i-1}.$ Notice that both
$\kappa$ and $\tau$ are normal colorings of the superedge $B$ which use only
colors $1,$ $2$ and $3.$

Next, for any other triple $\sigma_{t}(u_{2i})\in T,$ the coloring
$\tilde{\sigma}_{2i}$ of $M_{2i}$ will on $\mathcal{B}_{2i}$ be defined as a
color permutation of $\kappa$, and on $\mathcal{B}_{2i-1}$ as a modification
of $\tau$ obtained by color permutation and/or color swapping along the Kempe
chain $P$. Namely, if $\sigma_{t}(u_{2i})=(t_{1},t_{2},t_{3})\in T,$ then
$\left.  \tilde{\sigma}_{2i}\right\vert _{\mathcal{B}_{2i}}=\kappa^{\prime}$
where $\kappa^{\prime}$ is a $3$-edge-coloring of $B$ obtained from $\kappa$
by replacing color $i$ by $t_{i}$ for $i=1,2,3.$ It remains to define the
restriction of $\tilde{\sigma}_{2i}$ to $\mathcal{B}_{2i-1}.$ We distinguish
the following cases illustrated by Figure \ref{Fig_tenDoubleStars}:

\begin{itemize}
\item[a)] if $\sigma_{t}(u_{2i})=(1,2,3)$, then $\left.  \tilde{\sigma}%
_{2i}\right\vert _{\mathcal{B}_{2i-1}}=\tau;$

\item[b)] if $\sigma_{t}(u_{2i})=(3,2,1)$, then $\left.  \tilde{\sigma}%
_{2i}\right\vert _{\mathcal{B}_{2i-1}}=\tau^{\prime}$ where $\tau^{\prime}$ is
obtained from $\tau$ by swapping colors $1$ and $3$ along $P;$

\item[c)] if $\sigma_{t}(u_{2i})=(4,2,5)$, then $\left.  \tilde{\sigma}%
_{2i}\right\vert _{\mathcal{B}_{2i-1}}=\tau^{\prime}$ where $\tau^{\prime}$ is
obtained from $\tau$ by replacing $1$ with $4$ and $3$ with $5$ along $P;$

\item[d)] if $\sigma_{t}(u_{2i})=(2,1,3)$, then $\left.  \tilde{\sigma}%
_{2i}\right\vert _{\mathcal{B}_{2i-1}}=\tau^{\prime}$ where $\tau^{\prime}$ is
obtained from $\tau$ by swapping $1$ and $2$ in whole $\mathcal{B}_{2i-1};$

\item[e)] if $\sigma_{t}(u_{2i})=(3,1,2)$, then $\left.  \tilde{\sigma}%
_{2i}\right\vert _{\mathcal{B}_{2i-1}}=\tau^{\prime}$ where $\tau^{\prime}$ is
obtained from $\tau$ by first swapping $1$ and $2$ in whole $\mathcal{B}%
_{2i-1}$ and then swapping $2$ and $3$ along $P$;

\item[f)] if $\sigma_{t}(u_{2i})=(4,1,5)$, then $\left.  \tilde{\sigma}%
_{2i}\right\vert _{\mathcal{B}_{2i-1}}=\tau^{\prime}$ where $\tau^{\prime}$ is
obtained from $\tau$ by first swapping $1$ and $2$ in $\mathcal{B}_{2i-1}$ and
then replacing $2$ by $4$ and $3$ by $5$ along $P;$

\item[g)] if $\sigma_{t}(u_{2i})=(1,5,2)$, then $\left.  \tilde{\sigma}%
_{2i}\right\vert _{\mathcal{B}_{2i-1}}=\tau^{\prime}$ where $\tau^{\prime}$ is
obtained from $\tau$ by first replacing $1$ with $4$ and $2$ with $5$ in
$\mathcal{B}_{2i-1},$ and then replacing $3$ with $2$ and $4$ with $1$ along
$P$;

\item[h)] if $\sigma_{t}(u_{2i})=(2,5,1)$, then $\left.  \tilde{\sigma}%
_{2i}\right\vert _{\mathcal{B}_{2i-1}}=\tau^{\prime}$ where $\tau^{\prime}$ is
obtained from $\tau$ by first replacing $1$ with $4$ and $2$ with $5$ in
$\mathcal{B}_{2i-1},$ and then replacing $3$ with $1$ and $4$ with $2$ along
$P$;

\item[i)] if $\sigma_{t}(u_{2i})=(4,5,3)$, then $\left.  \tilde{\sigma}%
_{2i}\right\vert _{\mathcal{B}_{2i-1}}=\tau^{\prime}$ where $\tau^{\prime}$ is
obtained from $\tau$ by replacing $1$ with $4$ and $2$ with $5$ in
$\mathcal{B}_{2i-1}$;

\item[j)] if $\sigma_{t}(u_{2i})=(3,5,4)$, then $\left.  \tilde{\sigma}%
_{2i}\right\vert _{\mathcal{B}_{2i-1}}=\tau^{\prime}$ where $\tau^{\prime}$ is
obtained from $\tau$ by first replacing $1$ with $4$ and $2$ with $5$ in
$\mathcal{B}_{2i-1}$, and then swapping colors $3$ and $4$ along $P.$
\end{itemize}

\noindent Notice that $\tilde{\sigma}_{2i}$ is $\tilde{\sigma}_{\mathrm{ext}}%
$-compatible since we have chosen $\sigma_{t}(u_{2i})\in T,$ where $T$
contains only those colors of edges incident to $u_{2i}$ which are compatible
with $\sigma_{t}(u_{2i-2})$ in $G,$ and then we colored (semi)edges shared by
$M_{2i}$ and $M_{\mathrm{ext}}$ the same way in $\tilde{\sigma}_{2i}$ as in
$\tilde{\sigma}_{\mathrm{ext}}$, as can bee seen by comparing Figures
\ref{Fig_tenDoubleStars} and \ref{Fig_rotation}. Also, it is obvious from
Figure \ref{Fig_tenDoubleStars} that $\tilde{\sigma}_{2i}$ is right-side
monochromatic for every $\sigma_{t}(u_{2i})\in T.$

\begin{figure}[h]
\begin{center}%
\begin{tabular}
[t]{ll}%
a) & \raisebox{-0.9\height}{\includegraphics[scale=0.45]{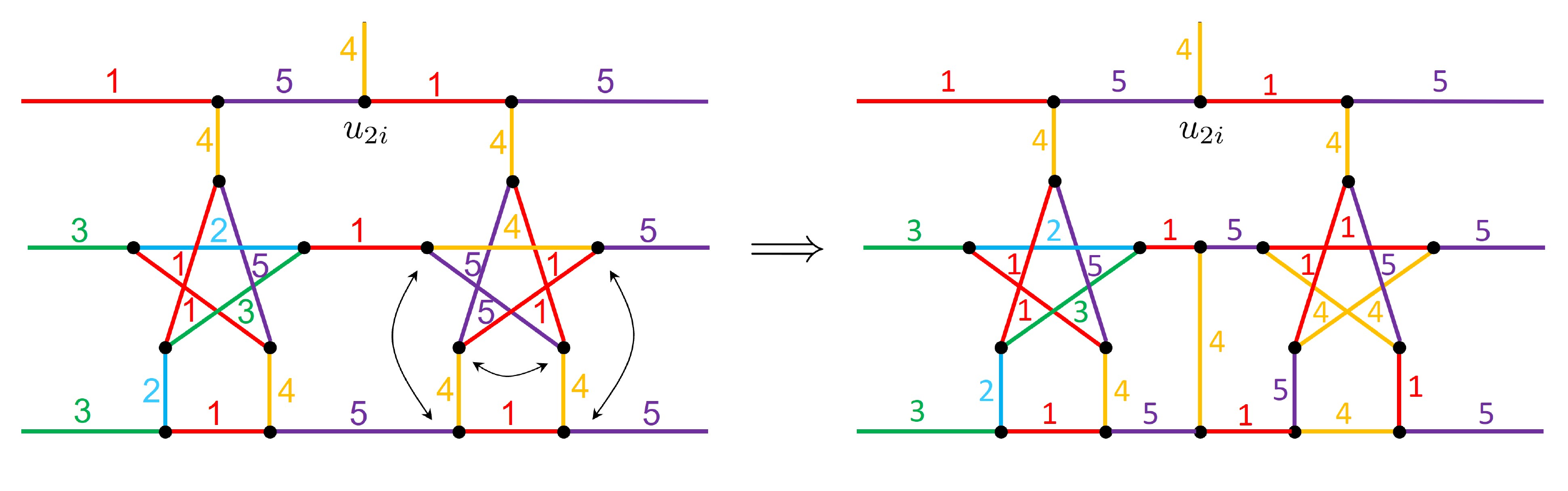}}\\
b) & \raisebox{-0.9\height}{\includegraphics[scale=0.45]{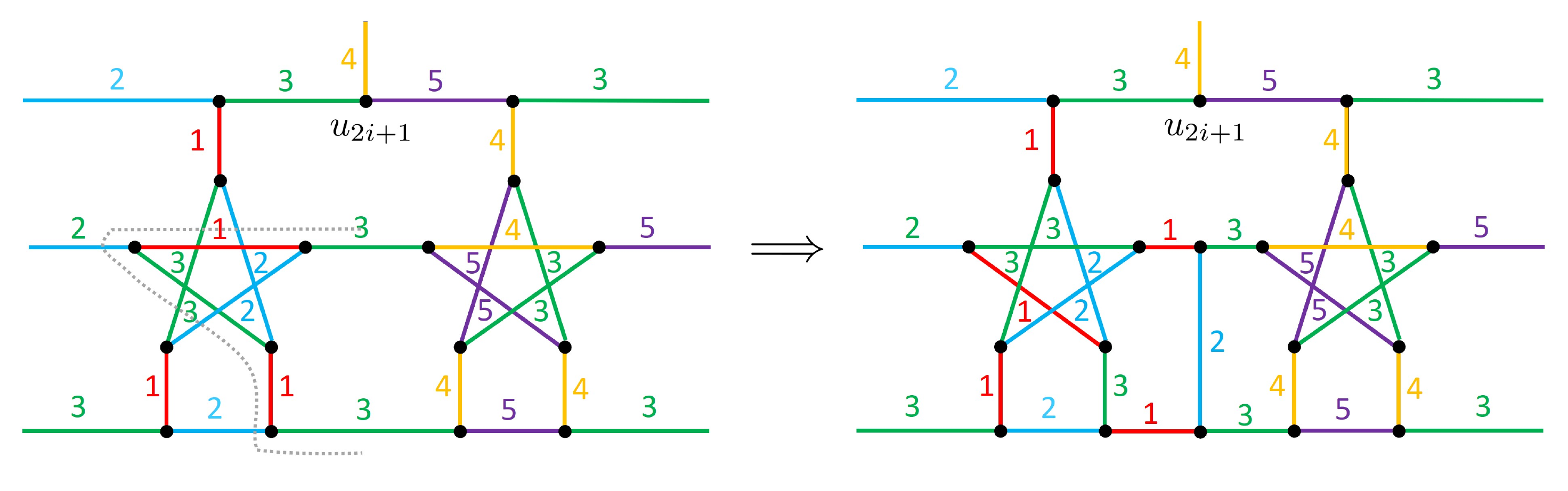}}
\end{tabular}
\end{center}
\caption{A modification of the coloring $\tilde{\sigma}_{A}$ when: a)
$\mathcal{A}_{2i}=A^{\prime}$, b) $\mathcal{A}_{2i+1}=A^{\prime}.$ In case a)
the ispomorphism $I$ is applied to $\left.  \tilde{\sigma}_{0}\right\vert
_{\mathcal{B}_{2i}},$ and in the case b) colors $1$ and $3$ are swapped along
the Kempe $(1,3)$-chain $Q$ in $\mathcal{B}_{2i-1}.$}%
\label{Fig_modification}%
\end{figure}

It remains to establish that every $\tilde{\sigma}_{2i}$ is also left-side
compatible, where for the semiedge of $M_{2i}$ incident to $a_{2i-1}$ this
follows from $\tilde{\sigma}_{2i}$ being $\tilde{\sigma}_{\mathrm{ext}}%
$-compatible, and for the remaining two semiedges on the left side this is
obvious from Figure \ref{Fig_tenDoubleStars} (it is sufficient to check the
compatibility of the left side of $M_{2i}$ for each case with the right-side
of case a), since in the case a) it holds that $\sigma_{t}(u_{2i})=(1,2,3)$ as
we assumed for $\sigma_{t}(u_{2i-2})$). Thus, we conclude that $\tilde{\sigma
}_{2i}$ defined in this way is indeed $\tilde{\sigma}_{\mathrm{ext}}%
$-compatible, right-side monochromatic and left-side compatible normal
$5$-coloring of a multipole $M_{2i}$.

Finally, recall that we assumed $\sigma_{t}(u_{2i-2})=(1,2,3),$ and then we
obtained the desired coloring of $M_{2i}$ under that assumption. It remains to
note that in the case of $\sigma_{t}(u_{2i-2})=(t_{1},t_{2},t_{3}),$ the
desired coloring of $M_{2i}$ is obtained by further applying a color
permutation which maps color $i$ to $t_{i}$ for $i=1,2,3.$ This concludes the
proof of Claim B.\medskip

We are now in a position to define a normal $5$-coloring $\tilde{\sigma}_{A}$
of $\tilde{G}_{A}.$ Namely, $\tilde{\sigma}_{A}$ is defined by $\left.
\tilde{\sigma}_{A}\right\vert _{M_{\mathrm{ext}}}=\tilde{\sigma}%
_{\mathrm{ext}}$ and $\left.  \tilde{\sigma}_{A}\right\vert _{M_{2i}}%
=\tilde{\sigma}_{2i}$ for every $i=0,\ldots,g/2.$ In order to establish that
$\tilde{\sigma}_{A}$ is well defined normal $5$-coloring, the edges on the
border of $M_{\mathrm{ext}}$ and each $M_{2i}$ have to be considered, and also
edges which connect $M_{2i-2}$ and $M_{2i}$ for every $i=0,\ldots,g/2.$ Since
every $\tilde{\sigma}_{2i}$ is $\tilde{\sigma}_{\mathrm{ext}}$-compatible,
$\left.  \tilde{\sigma}_{A}\right\vert _{M_{\mathrm{ext}}\cup M_{2i}}$ is well
defined and normal for every $i.$ As for the edges connecting $M_{2i-2}$ and
$M_{2i},$ notice that we assumed $\tilde{\sigma}_{2i-2}$ is right-side
monochromatic. Since we also assumed that $\tilde{\sigma}_{2i}$ is left-side
compatible, by definition this implies that $\tilde{\sigma}_{2i}$ is
$\tilde{\sigma}_{2i-2}$-compatible, so $\left.  \tilde{\sigma}_{A}\right\vert
_{M_{2i-2}\cup M_{2i}}$ is also well defined and normal. Further,
$\tilde{\sigma}_{A}$ is proper because $\tilde{\sigma}_{\mathrm{ext}}$ and
$\tilde{\sigma}_{2i}$ are proper, and it uses the same colors as
$\tilde{\sigma}_{\mathrm{ext}}$ and $\tilde{\sigma}_{2i}.$ We conclude that
$\tilde{\sigma}_{A}$ is indeed a normal $5$-coloring of $\tilde{G}_{A}.$

It remains to prove that any other graph $\tilde{G}$ from $\mathcal{G}%
_{C}(\mathcal{A},\mathcal{B}),$ beside $\tilde{G}_{A},$ also has a normal
$5$-coloring. Recall that in $\tilde{G}_{A}$ it holds that $\mathcal{A}_{i}=A$
for every vertex $u_{i}$ of $C,$ and in $\tilde{G}$ it may hold $\mathcal{A}%
_{i}=A^{\prime}$ for some vertices $u_{i}.$ We will construct a normal
$5$-coloring of $\tilde{G}$ by slightly modifying the normal coloring
$\tilde{\sigma}_{A}$.

\bigskip\noindent\textbf{Claim C.} \emph{There exists a normal }%
$5$\emph{-coloring }$\tilde{\sigma}$\emph{ of }$\tilde{G}.$

\medskip\noindent Assume vertices $u_{i_{1}},\ldots,u_{i_{k}}$ from $C$ are
superposed by $A^{\prime}$ in $\tilde{G},$ and all other vertices of $C$ are
superposed by $A.$ We may assume $i_{1}<\cdots<i_{k}.$ Vertices which
subdivide edges $d_{i_{j}-1}c_{i_{j}}$ and $h_{i_{j}-1}g_{i_{j}}$ are denoted
by $u_{i_{j}}^{\prime}$ and $u_{i_{j}}^{\prime\prime}$ as in Figure
\ref{Fig_Kochol}. Denote by $\tilde{G}_{j}$ a graph from $\mathcal{G}%
_{C}(\mathcal{A},\mathcal{B})$ in which only vertices $u_{i_{1}}%
,\ldots,u_{i_{j}}$ are superposed by $A^{\prime},$ and all other by $A,$ so
obviously $\tilde{G}=\tilde{G}_{k}.$ We may also denote $\tilde{G}_{0}%
=\tilde{G}_{A}$ and $\tilde{\sigma}_{0}=\tilde{\sigma}_{A}.$ We will construct
a normal $5$-coloring $\tilde{\sigma}_{k}$ of $\tilde{G}_{k}$ inductively,
i.e. for every $j=1,\ldots,k$ we will construct $\tilde{\sigma}_{j}$ of
$\tilde{G}_{j}$ from $\tilde{\sigma}_{j-1}$ of $\tilde{G}_{j-1}.$ Without loss
of generality we may assume that for $j=1$ the index $i_{j}=i_{1}$ is even.

\bigskip\textbf{Case C.1: }$j=1.$ Since $i_{1}$ is even (here consider Figure
\ref{Fig_modification}.a) for illustration), it holds that $\left.
\tilde{\sigma}_{0}\right\vert _{\mathcal{B}_{i_{1}}}$ is equal to $\kappa$ up
to color permutation, so it uses only $3$ colors $t_{1},$ $t_{2}$ and $t_{3}.$
We may assume that colors $t_{i}$ of $\tilde{\sigma}_{0}$ are denoted so that
$\sigma^{t}(u_{i_{1}})=(t_{1},t_{2},t_{3}).$ For example, in Figure
\ref{Fig_modification}.a), we have $\sigma^{t}(u_{2i})=(4,5,1)$. The
edge-coloring $\tilde{\sigma}_{1}^{\prime}$ of $\mathcal{B}_{i_{1}}$ is
defined by $\tilde{\sigma}_{1}^{\prime}=I(\left.  \tilde{\sigma}%
_{0}\right\vert _{\mathcal{B}_{i_{1}}}).$ Since $\left.  \tilde{\sigma}%
_{0}\right\vert _{\mathcal{B}_{i_{1}}}$ is normal and uses only colors
$t_{1},t_{2},t_{3},$ the same holds for $\tilde{\sigma}_{1}^{\prime}.$ Also,
semiedges of $\mathcal{B}_{i_{1}}$ incident to vertices $c_{i_{j}}$ and
$g_{i_{j}}$ are in $\tilde{\sigma}_{1}^{\prime}$ colored\ by $t_{3}$ and
$t_{2},$ respectively. We now define $\tilde{\sigma}_{1}$ of $\tilde{G}_{1}$
by%
\[
\tilde{\sigma}_{1}(x)=\left\{
\begin{array}
[c]{ll}%
t_{3} & \text{if }x=u_{i_{1}}^{\prime}u_{i_{1}}^{\prime\prime},\\
\tilde{\sigma}_{1}^{\prime}(x) & \text{if }x\text{ belongs to }\mathcal{B}%
_{i_{1}},\\
\tilde{\sigma}_{0}(x) & \text{otherwise.}%
\end{array}
\right.
\]

Let us first establish that $\tilde{\sigma}_{1}$ is well defined, i.e. that
$\tilde{\sigma}_{1}^{\prime}$ of $\mathcal{B}_{i_{1}}$ is compatible with
$\tilde{\sigma}_{0}$ restricted to the rest of $\tilde{G}.$ Recall that
$\tilde{\sigma}_{0}$ is right-side monochromatic on the multipole $M_{i_{1}}$
induced by $V(\mathcal{B}_{i_{1}-1})\cup V(\mathcal{A}_{i_{1}})\cup
V(\mathcal{B}_{i_{1}}).$ Consequently, the extended coloring of semiedges
incident to $d_{i_{1}}$ and $h_{i_{1}}$ is the same in $\tilde{\sigma}%
_{1}^{\prime}$ and $\tilde{\sigma}_{0},$ which assures that $\tilde{\sigma
}_{1}^{\prime}$ is compatible with $\left.  \tilde{\sigma}_{0}\right\vert
_{\mathcal{B}_{i_{1}+1}}.$ Next, we need to establish that $\tilde{\sigma}%
_{1}$ is proper. For all vertices distinct from $u_{i_{1}}^{\prime}$ and
$u_{i_{1}}^{\prime\prime},$ the properness of $\tilde{\sigma}_{1}$ is the
consequence of $\tilde{\sigma}_{0}$ and $\tilde{\sigma}_{1}^{\prime}$ being
proper. For $u_{i_{1}}^{\prime}$ and $u_{i_{1}}^{\prime\prime},$ notice that
$\tilde{\sigma}_{0}(d_{i_{1}-1}c_{i_{1}})=t_{3}$ and $\tilde{\sigma}%
_{0}(h_{i_{1}-1}g_{i_{1}})=t_{2},$ so the properness of $\tilde{\sigma}_{1}$
at vertices $u_{i_{1}}^{\prime}$ and $u_{i_{1}}^{\prime\prime}$ follows from
$t_{2}\not =t_{3}$, the properties of the isomorphism $I$ and the fact that
$\tilde{\sigma}_{1}(u_{i_{1}}^{\prime}u_{i_{1}}^{\prime\prime})=t_{3}.$

Further, notice that $\tilde{\sigma}_{1}$ is obviously a $5$-coloring since it
uses the same colors as $\tilde{\sigma}_{0},$ so it remains to establish that
$\tilde{\sigma}_{1}$ is normal. If an edge $e$ of $\tilde{G}$ is not incident
to $u_{i_{1}}^{\prime}$ nor $u_{i_{1}}^{\prime\prime},$ then $e$ is normal by
$\tilde{\sigma}_{1}$ since it is normal by $\tilde{\sigma}_{0}$ or
$\tilde{\sigma}_{1}^{\prime}.$ Since $\left.  \tilde{\sigma}_{1}\right\vert
_{\mathcal{B}_{i_{1}}}$ uses only colors $t_{1}$ , $t_{2}$ and $t_{3},$ then
$\tilde{\sigma}_{1}(d_{i_{1}-1}u_{i_{1}}^{\prime})=t_{2}$ and $\tilde{\sigma
}_{1}(h_{i_{1}-1}u_{i_{1}}^{\prime\prime})=t_{3}$ implies that edges
$u_{i_{1}}^{\prime}u_{i_{1}}^{\prime\prime},$ $u_{i_{1}}^{\prime}c_{i_{1}}$
and $u_{i_{1}}^{\prime}g_{i_{1}}$ are poor by $\tilde{\sigma}_{1}.$ As for
edges $u_{i_{1}}^{\prime}d_{i_{1}-1}$ and $u_{i_{1}}^{\prime}h_{i_{1}-1},$
their color and the colors incident to their end-vertices is the same by
$\tilde{\sigma}_{0}$ and $\tilde{\sigma}_{1},$ so they are normal by
$\tilde{\sigma}_{1}$ since they are normal by $\tilde{\sigma}_{0}.$

\bigskip\textbf{Case C.2:} $j>1.\ $If $i_{j}$ is even, then $\tilde{\sigma
}_{j}$\ is constructed from $\tilde{\sigma}_{j-1}$ in the same way as in Case
C.1. Hence, we may assume $i_{j}$ is odd (here consider Figure
\ref{Fig_modification}.b) for illustration). In this case it holds that
$\left.  \tilde{\sigma}_{j-1}\right\vert _{\mathcal{B}_{i_{j}-1}}$ is equal to
$\kappa$ up to color permutation. Thus, $\left.  \tilde{\sigma}_{j-1}%
\right\vert _{\mathcal{B}_{i_{j}-1}}$ uses only three colors $t_{1},t_{2}$ and
$t_{3},$ where we may assume that colors are denoted so that $\sigma
^{t}(u_{i_{j}-1})=(t_{1},t_{2},t_{3}),$ so $Q$ is a Kempe $(t_{2},t_{1}%
)$-chain in $\mathcal{B}_{i_{j}-1}$ with respect to $\tilde{\sigma}_{0}$. For
example, in Figure \ref{Fig_modification}.b) we have $\sigma^{t}%
(u_{2i-1})=(1,3,2),$ and $Q$ is a $(3,1)$-Kempe chain in $\mathcal{B}%
_{i_{j}-1}$. Notice that $I(Q)$ is Kempe $(t_{2},t_{1})$-chain with respect to
$I(\tilde{\sigma}_{j-1})$, too. We define $\tilde{\sigma}_{j}^{\prime}$ of
$\mathcal{B}_{i_{j}-1}$ as a coloring obtained from $\left.  \tilde{\sigma
}_{j-1}\right\vert _{\mathcal{B}_{i_{j}-1}}$ by swapping colors $t_{1}$ and
$t_{2}$ along $Q.$ Since $\left.  \tilde{\sigma}_{j-1}\right\vert
_{\mathcal{B}_{i_{j}-1}}$ uses only $3$ colors $t_{1},$ $t_{2}$ and $t_{3},$
it follows that $\tilde{\sigma}_{j}^{\prime}$ uses the same $3$ colors.
Observation \ref{Obs_Kempe} implies that $\tilde{\sigma}_{j}^{\prime}$ is
normal. Now we define an edge-coloring $\tilde{\sigma}_{{}}$ of $\tilde{G}%
_{j}$ by
\[
\tilde{\sigma}_{j}(x)=\left\{
\begin{array}
[c]{ll}%
t_{3} & \text{if }x=u_{i_{1}}^{\prime}u_{i_{1}}^{\prime\prime},\\
\tilde{\sigma}_{j}^{\prime}(x) & \text{if }x\text{ belongs to }\mathcal{B}%
_{i_{j}-1},\\
\tilde{\sigma}_{j-1}(x) & \text{otherwise.}%
\end{array}
\right.
\]
Obviously, $\tilde{\sigma}_{j}$ is well defined $5$-coloring of $\tilde{G}%
_{j}.$ It remains to establish that $\tilde{\sigma}_{j}$ is proper and normal.
To establish that $\tilde{\sigma}_{j}$ is proper, it is sufficient to consider
vertices $u_{i_{j}}^{\prime}$ and $u_{i_{j}}^{\prime\prime},$ since for all
other vertices the properness follows from $\tilde{\sigma}_{j}^{\prime}$ and
$\tilde{\sigma}_{j-1}$ being proper. By swapping colors $t_{1}$ and $t_{2}$
along $Q$ in $\mathcal{B}_{i_{1}-1}$ and defining $\tilde{\sigma}_{j}%
(u_{i_{j}}^{\prime}u_{i_{j}}^{\prime\prime})=t_{3},$ we obtained
$\tilde{\sigma}_{j}(u_{i_{j}}^{\prime})=\tilde{\sigma}_{j}(u_{i_{j}}%
^{\prime\prime})=\{t_{1},t_{2},t_{3}\},$ so $\tilde{\sigma}_{j}$ is proper.

The same fact that $\tilde{\sigma}_{j}(u_{i_{j}}^{\prime})=\tilde{\sigma}%
_{j}(u_{i_{j}}^{\prime\prime})=\{t_{1},t_{2},t_{3}\}$ together with $\left.
\tilde{\sigma}_{j}\right\vert _{\mathcal{B}_{i_{j}-1}}$ being a normal
coloring which uses only colors $t_{1},$ $t_{2}$ and $t_{3},$ implies that
edges $u_{i_{j}}^{\prime}u_{i_{j}}^{\prime\prime},$ $u_{i_{j}}^{\prime
}d_{i_{j}-1}$ and $u_{i_{j}}^{\prime\prime}h_{i_{j}}$ are poor by
$\tilde{\sigma}_{j}.$ As for edges $u_{i_{j}}^{\prime}c_{i_{j}}$ and
$u_{i_{j}}^{\prime\prime}g_{i_{j}},$ notice that their color and the colors
incident to their end-vertices by $\tilde{\sigma}_{j}$ in $\tilde{G}_{j}$ is
the same as the color of edges $d_{i_{j}-1}c_{i_{j}}$ and $h_{i_{j}-1}%
g_{i_{j}}$, respectively, by $\tilde{\sigma}_{j-1}$ in $\tilde{G}_{j-1}.$
Thus, $u_{i_{j}}^{\prime}c_{i_{j}}$ and $u_{i_{j}}^{\prime\prime}g_{i_{j}}$
are normal by $\tilde{\sigma}_{j}$ in $\tilde{G}_{j}$ since $\tilde{\sigma
}_{j-1}$ is normal. Thus, the case when $j>1$ is odd is resolved, which
concludes the proof of Claim C.\medskip

\bigskip\noindent\textbf{Claim D.} \emph{The coloring }$\tilde{\sigma}$
\emph{of }$\tilde{G}$ \emph{contains at least }$18$\emph{ poor edges.}

\medskip\noindent Recall that the coloring $\kappa$ of the superedge $B$ is
defined as $\kappa=\left.  \tilde{\sigma}_{2i}\right\vert _{\mathcal{B}_{2i}}$
for $\tilde{\sigma}_{2i}$ from Figure \ref{Fig_tenDoubleStars}.a). Notice that
$B$ has $9$ edges, every of those $9$ edges is poor by $\kappa,$ and it will
remain poor if color permutation is applied to $\kappa.$ Our construction of
the coloring $\tilde{\sigma}$ of $\tilde{G}$ implies that $\tilde{\sigma
}(\mathcal{B}_{2i})$ is equal to $\kappa$ or to $I(\kappa)$ up to color
permutation, for every $i=0,\ldots,(g-2)/2.$ Since $G$ is simple, the shortest
even cycle $C$ of $G$ contains is of length $g\geq4,$ thus at least $2$
superedges $\mathcal{B}_{2i}$ are colored by $\kappa$, which implies that the
constructed coloring $\tilde{\sigma}$ of $\tilde{G}$ has at least $18$ poor
edges.\medskip

Hence, we have established that every superposition $\tilde{G}\in
\mathcal{G}_{C}(\mathcal{A},\mathcal{B})$ has a normal $5$-coloring. Moreover,
the constructed normal coloring introduces at least $18$ poor edges, and the
claim of the theorem is established.
\end{proof}

\begin{figure}[h]
\begin{center}
\includegraphics[scale=0.6]{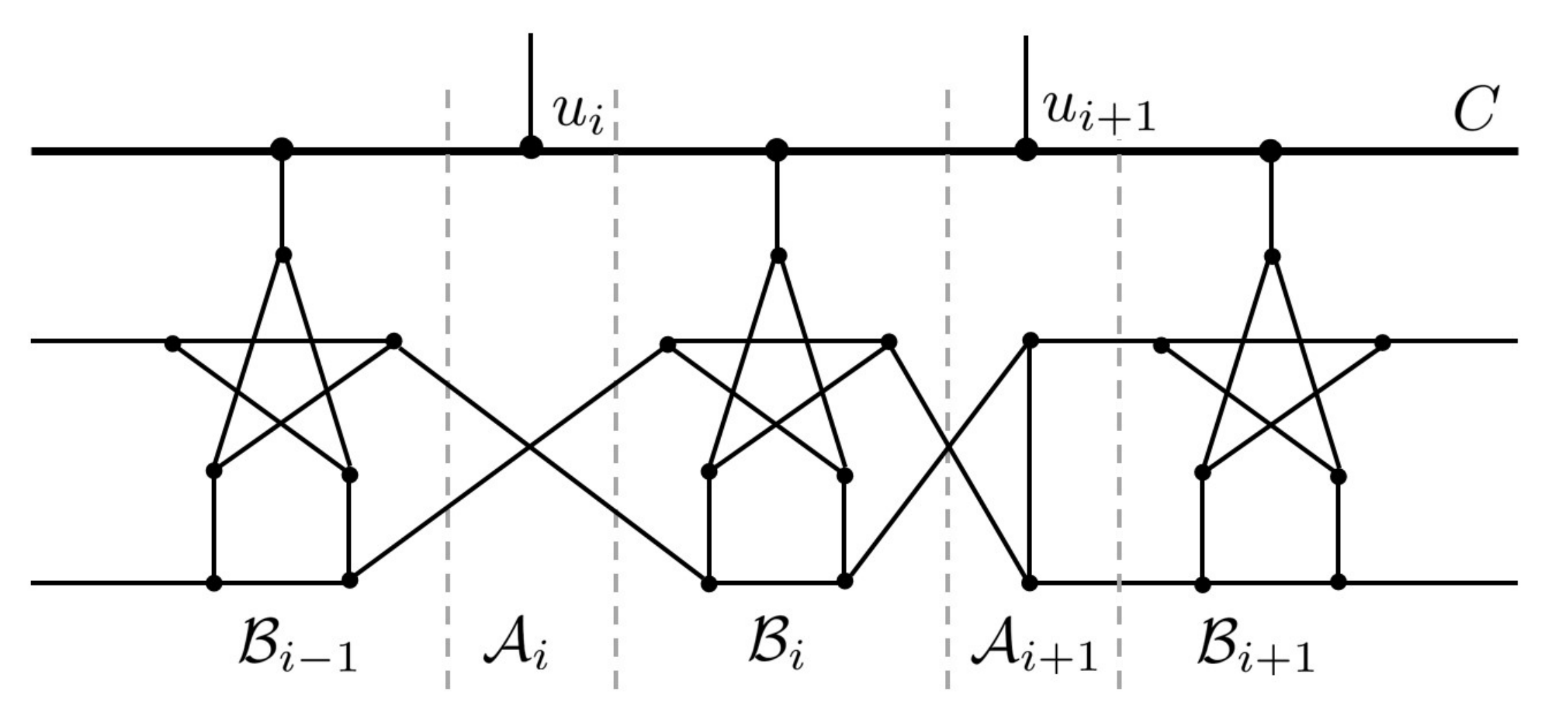}
\end{center}
\caption{Figure shows a twist in supervertices $\mathcal{A}_{i}$ of type $A$
and $\mathcal{A}_{i+1}$ of type $A^{\prime}.$}%
\label{Fig_twist}%
\end{figure}

In Theorem \ref{Tm_main} we considered only snarks from $\mathcal{G}%
_{C}(\mathcal{A},\mathcal{B})$ in which semiedges of $\mathcal{B}_{i-1}$ and
$\mathcal{B}_{i}$ are identified in one particular way, the one from Figure
\ref{Fig_Kochol}. The normal colorings constructed in the proof of Theorem
\ref{Tm_main} can easily be modified to be a normal coloring of a
superposition snark which has the so called twist in some of the
supervertices. First, we say that a supervertex $\mathcal{A}_{i}$ has a
\emph{twist}, if semiedges of $\mathcal{B}_{i-1}$ and $\mathcal{B}_{i}$ are
identified as in Figure \ref{Fig_twist}. Let $\mathcal{G}_{C}^{\mathrm{tw}%
}(\mathcal{A},\mathcal{B})$ denote the set of all superpositions in which some
of the supervertices $\mathcal{A}_{i}$ may have a twist.

\begin{corollary}
Let $G$ be a snark which admits a normal $5$-coloring and let $C$ be an even
cycle in $G$. Then every $\tilde{G}^{\mathrm{tw}}\in\mathcal{G}_{C}%
^{\mathrm{tw}}(\mathcal{A},\mathcal{B})$ admits a normal $5$-coloring which
has at least $18$ poor edges.
\end{corollary}

\begin{proof}
Notice that the isomorphism $I$ of the superedge $B$ can be used to reduce the
number of twists in $\tilde{G}^{\mathrm{tw}}.$ Namely, if $\tilde
{G}^{\mathrm{tw}}$ has even number of twists, i.e. $\mathcal{A}_{i_{j}}$ has a
twist for $j=1,\ldots,2k,$ then applying repeatedly isomorphism $I$ to
superedges $\mathcal{B}_{l}$ of $\tilde{G}^{\mathrm{tw}},$ for every
$l=i_{2j},\ldots,i_{2j+1}$ and $j=1,\ldots,k,$ one obtains that $\tilde
{G}^{\mathrm{tw}}$ is isomorphic to a superposition $\tilde{G}\in
\mathcal{G}_{C}(\mathcal{A},\mathcal{B})$ without twists, and then the claim
holds by Theorem \ref{Tm_main}. On the other hand, if $\tilde{G}^{\mathrm{tw}%
}$ has odd number of twists then the number of twists in a similar manner can
be reduced to only one. Hence, we can assume that $\tilde{G}^{\mathrm{tw}}$
has twist in only one supervertex $\mathcal{A}_{i}$, where without loss of
generality we may assume that $i$ is odd. For $\tilde{G}^{\mathrm{tw}},$ let
$\tilde{G}\in\mathcal{G}_{C}(\mathcal{A},\mathcal{B})$ be a corresponding
snark without the twist, and let $\tilde{\sigma}$ be the normal coloring of
$\tilde{G}$ as constructed in the proof of Theorem \ref{Tm_main}. Since
$\left.  \tilde{\sigma}\right\vert _{\mathcal{B}_{i-1}}$ is right-side
monochromatic, then the coloring obtained by rewiring edges of $\tilde{G}$ to
obtain the twist remains well defined, normal and all poor edges are preserved.
\end{proof}

\begin{figure}[h]
\begin{center}
\includegraphics[scale=0.7]{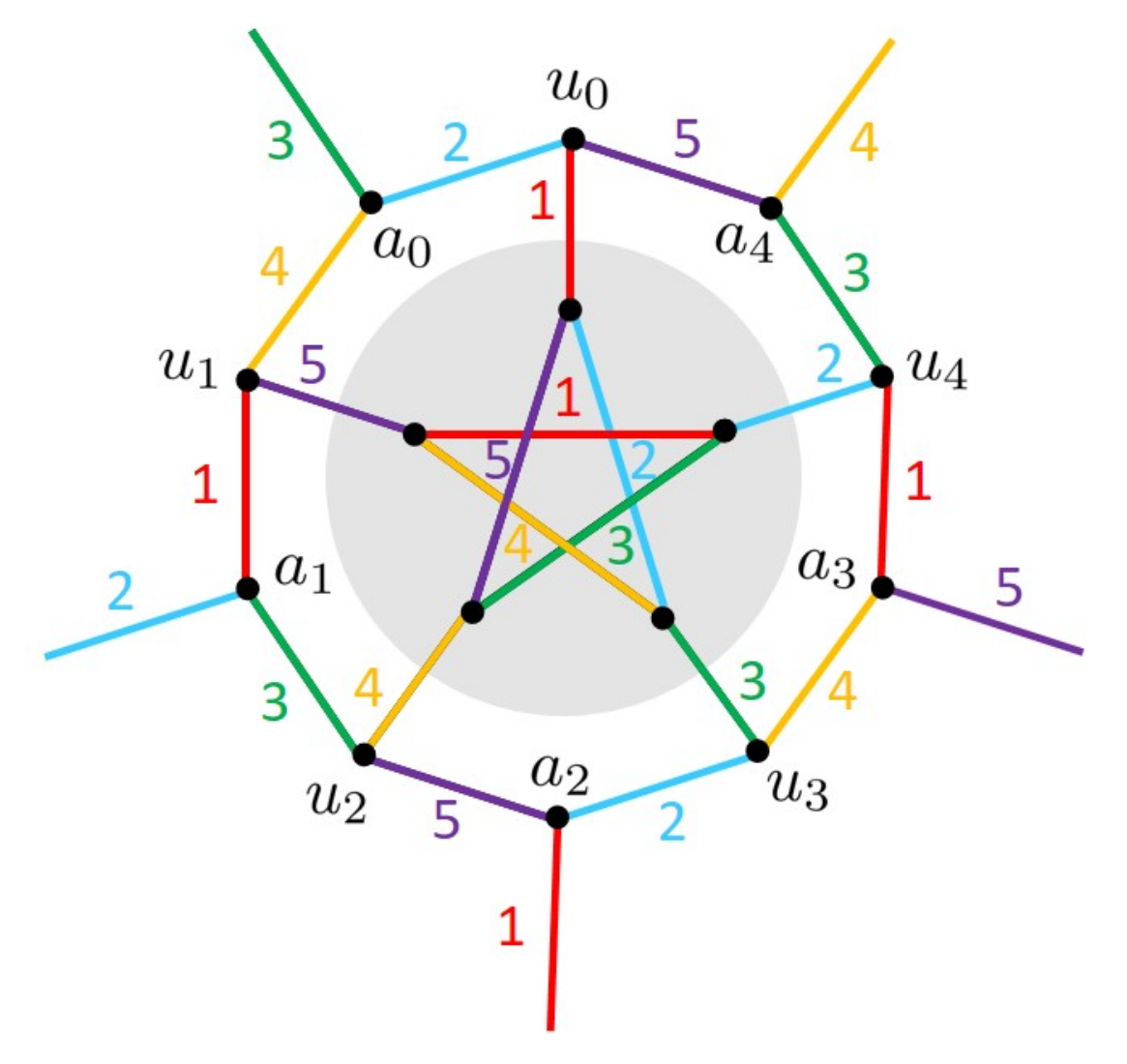}
\end{center}
\caption{Within the shadowed circle the restriction of the only normal
$5$-edge-coloring $\sigma$ of Petersen graph $G$ to $M_{\mathrm{int}}$ is
shown, and the whole figure shows the only normal $5$-edge-coloring
$\tilde{\sigma}_{\mathrm{ext}}$ of $M_{\mathrm{ext}}$ such that $\left.
\tilde{\sigma}_{\mathrm{ext}}\right\vert _{M_{\mathrm{int}}}=\left.
\sigma\right\vert _{M_{\mathrm{int}}}.$ Notice that the outer $10$-cycle in
$M_{\mathrm{ext}}$ arises from subdivisions of $5$-cycle in Petersen graph
when superposing edges of it by $B$.}%
\label{Fig_odd}%
\end{figure}

\paragraph{Case of odd cycles.}

The approach of Theorem \ref{Tm_main} cannot be extended to odd cycles in an
obvious way. To see this, we will assume the same notation as in the proof of
Theorem \ref{Tm_main}, i.e. by $M_{\mathrm{int}}$ we denote a submultipole of
$\tilde{G}$ induced by $V(G)\backslash V(C),$\ and by $M_{\mathrm{ext}}$ we
denote a submultipole of $\tilde{G}$ induced by $V(G)\cup\{a_{i}%
:i=0,\ldots,g-1\}.$ For a normal $5$-coloring $\sigma$ of a snark $G$ and an
even cycle $C$ in $G,$ let $\tilde{\sigma}$ be the normal $5$-coloring of
$\tilde{G}\in\mathcal{G}_{C}(\mathcal{A},\mathcal{B})$ as constructed in the
proof of Theorem \ref{Tm_main}. Notice that$\ \left.  \sigma\right\vert
_{M_{\mathrm{int}}}=\left.  \tilde{\sigma}\right\vert _{M_{\mathrm{int}}},$
i.e. we extended the coloring $\sigma$ of $G$ to the coloring $\tilde{\sigma}$
of $\tilde{G}$ by preserving colors of edges of $G$ outside of $C.$ In the
case of an odd cycle $C$, this is not always possible, as it is shown by the
following proposition which we verified in silico.

\begin{proposition}
\label{Prop_odd}Let $G$ be the Petersen graph, $\sigma$ the only normal
$5$-coloring of $G$ up to color permutation and $C$ a cycle of length $5$ in
$G.$ Let $\tilde{G}\in\mathcal{G}_{C}(\mathcal{A},\mathcal{B})$ be the snark
in which every vertex of $C$ is superposed by $A.$ Then there exists precisely
one normal coloring $\tilde{\sigma}_{\mathrm{ext}}$ of the submultipole
$M_{\mathrm{ext}}$ of $\tilde{G}$ (see Figure \ref{Fig_odd}) such that
$\left.  \tilde{\sigma}_{\mathrm{ext}}\right\vert _{M_{\mathrm{int}}}=\left.
\sigma\right\vert _{M_{\mathrm{int}}}.$ Also, for the submultipole $M_{123}$
of $\tilde{G}$ induced by $V(\mathcal{B}_{1})\cup V(\mathcal{A}_{2})\cup
V(\mathcal{B}_{2})\cup V(\mathcal{A}_{3})\cup V(\mathcal{B}_{3}),$ there
exists no normal $5$-coloring $\tilde{\sigma}_{1}$ of $M_{123}$ which is
compatible with $\tilde{\sigma}_{\mathrm{ext}}$.
\end{proposition}

Assume now that $\tilde{\sigma}$ is a normal $5$-coloring of the graph
$\tilde{G}$ from the above proposition. If we define $\tilde{\sigma
}_{\mathrm{ext}}=\left.  \tilde{\sigma}\right\vert _{M_{\mathrm{ext}}}$ and
$\tilde{\sigma}_{1}=\left.  \tilde{\sigma}\right\vert _{M_{123}},$ then
$\tilde{\sigma}_{1}$ would be compatible with $\tilde{\sigma}_{\mathrm{ext}}.$
This further implies that in order to construct a normal $5$-coloring
$\tilde{\sigma}$ of $\tilde{G},$ it cannot hold $\left.  \tilde{\sigma
}\right\vert _{M_{\mathrm{int}}}=\left.  \sigma\right\vert _{M_{\mathrm{int}}%
},$ i.e. one cannot simply take a normal coloring $\sigma$ of the Petersen
graph $G$ and then extend it to the superposition $\tilde{G}$ by preserving
colors outside the cycle $C.$

\begin{figure}[ph]
\begin{center}%
\begin{tabular}
[c]{c}%
\begin{tabular}
[t]{ll}%
a) & \raisebox{-0.9\height}{\includegraphics[scale=0.7]{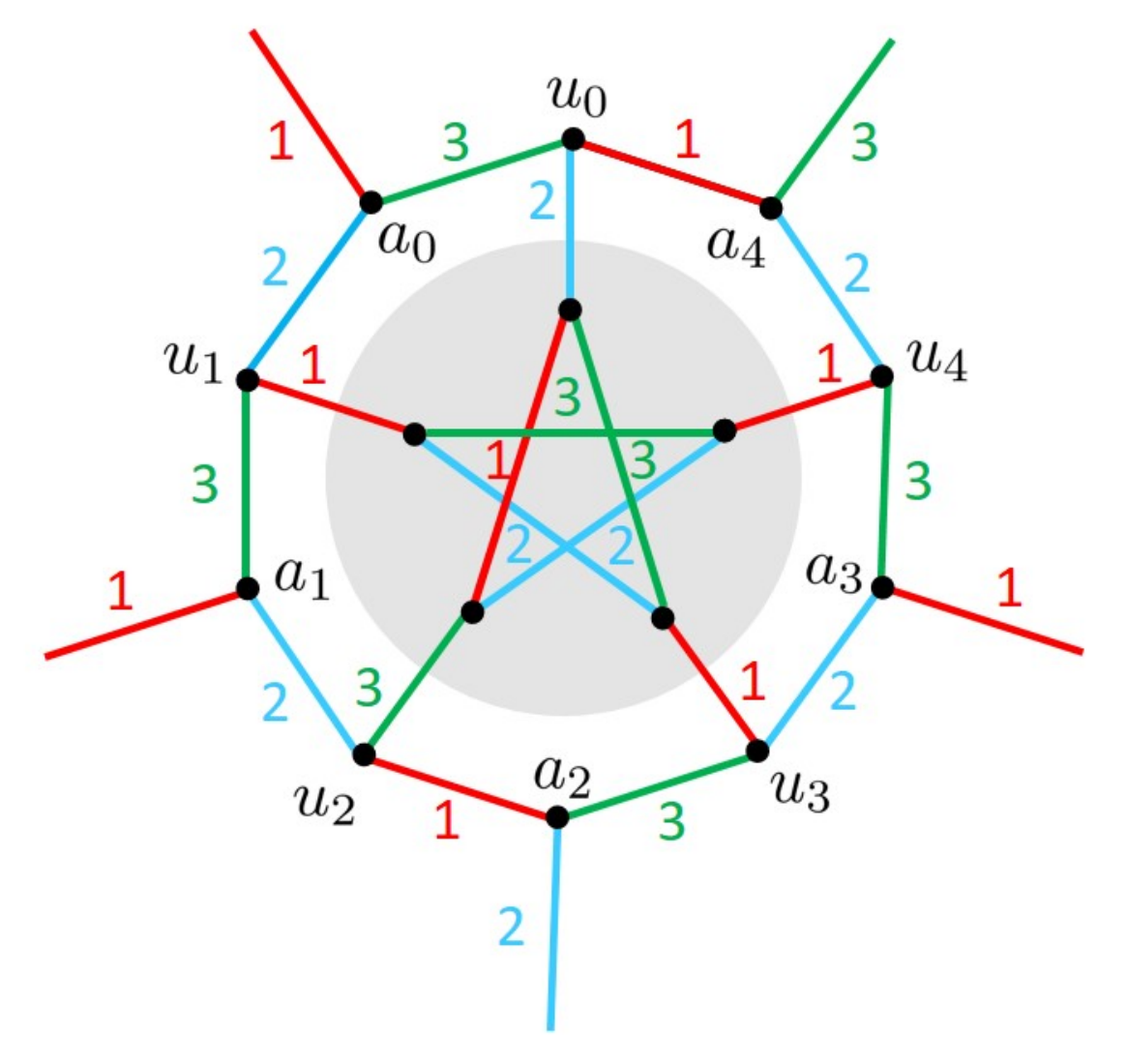}}
\end{tabular}
\\%
\begin{tabular}
[t]{llll}%
b) & \raisebox{-0.9\height}{\includegraphics[scale=0.6]{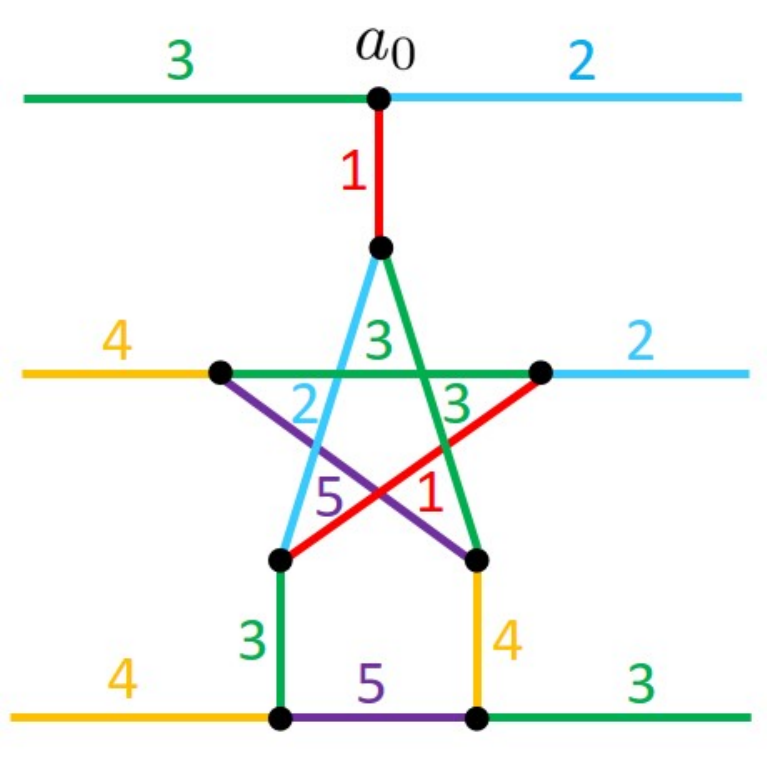}} & c) &
\raisebox{-0.9\height}{\includegraphics[scale=0.6]{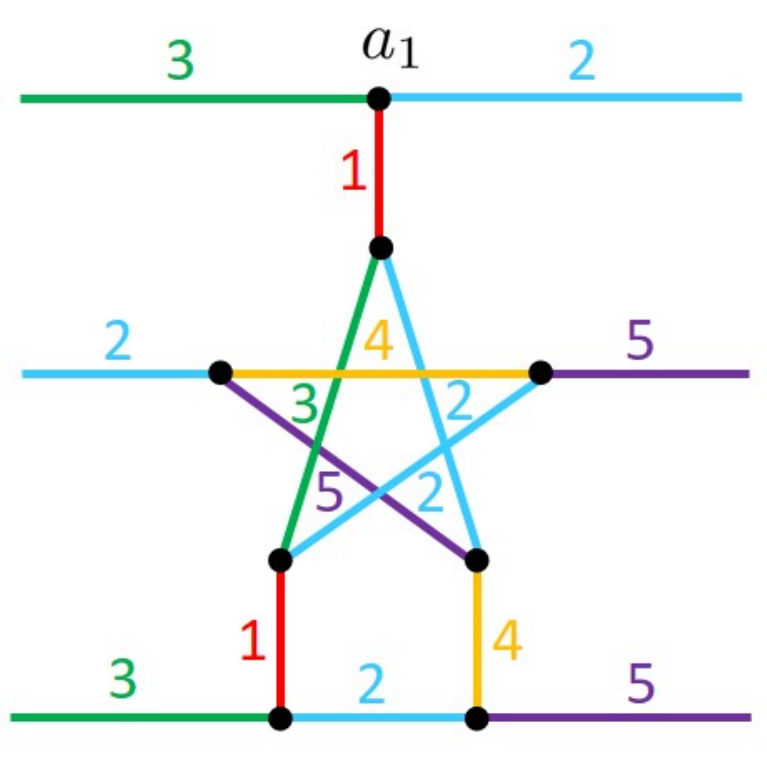}}
\end{tabular}
\\%
\begin{tabular}
[t]{llll}%
d) & \raisebox{-0.9\height}{\includegraphics[scale=0.6]{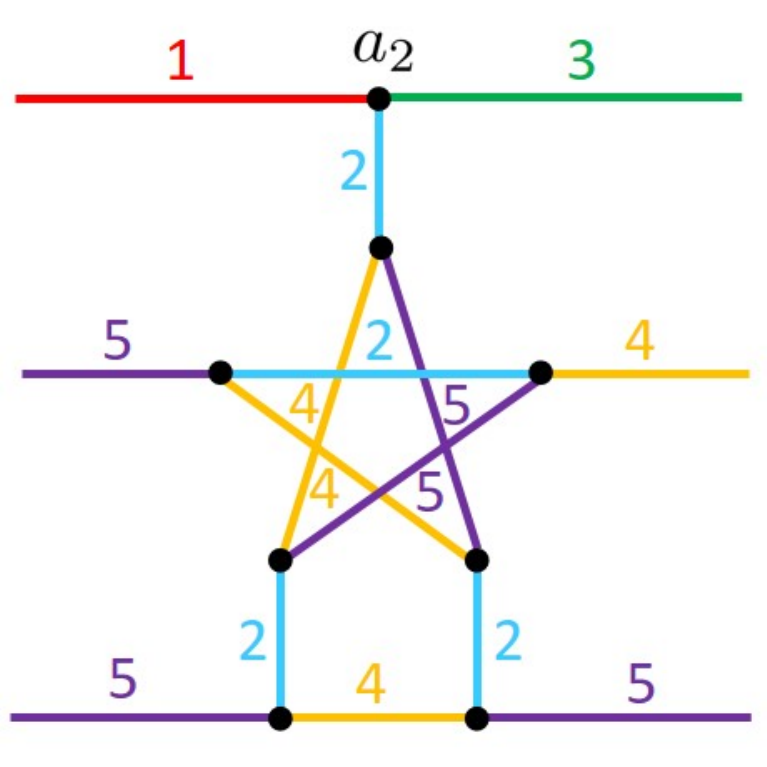}} & e) &
\raisebox{-0.9\height}{\includegraphics[scale=0.6]{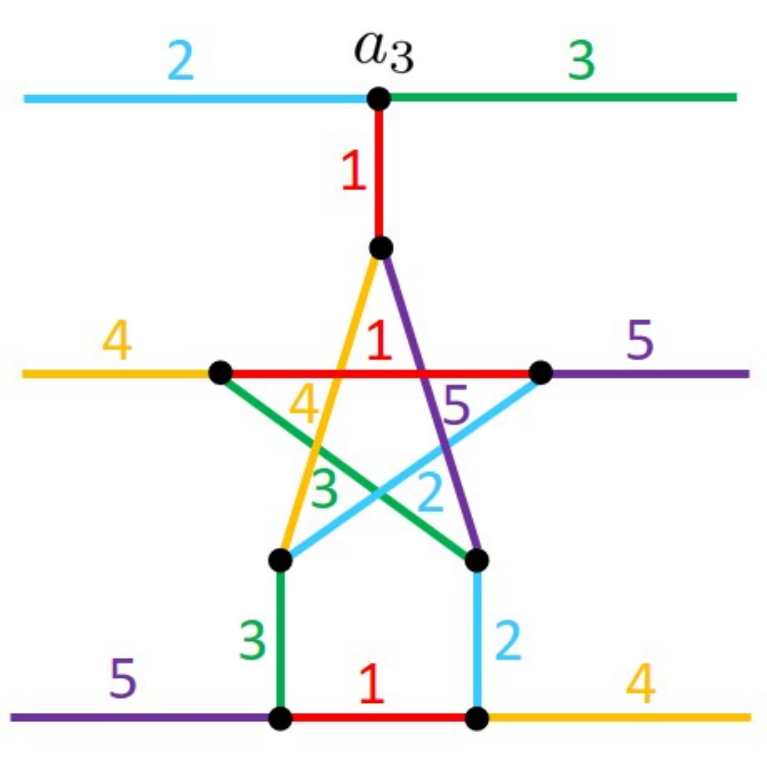}}
\end{tabular}
\\%
\begin{tabular}
[t]{ll}%
f) & \raisebox{-0.9\height}{\includegraphics[scale=0.6]{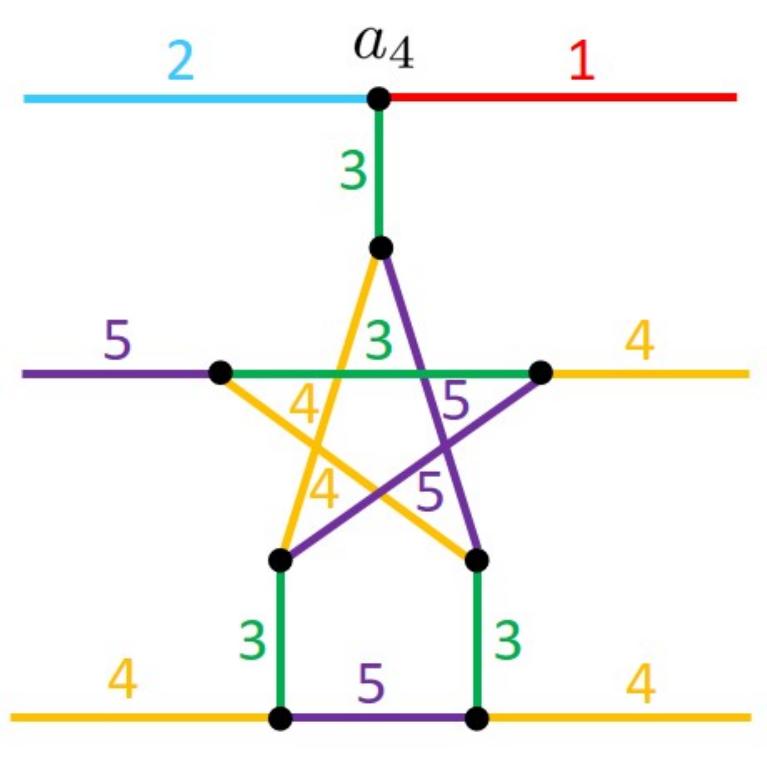}}
\end{tabular}
\end{tabular}
\end{center}
\caption{The coloring $\tilde{\sigma}$ of $\tilde{G}$ from Proposition
\ref{Prop_odd} shown as a restriction to: a) $M_{\mathrm{ext}},$ b)
$\mathcal{B}_{0},$ c) $\mathcal{B}_{1},$ d) $\mathcal{B}_{2},$ e)
$\mathcal{B}_{3},$ f) $\mathcal{B}_{4}.$}%
\label{Fig_oddSuperpositionColoring}%
\end{figure}

\begin{remark}
There exists at least one snark $G$ (the Petersen graph) and at least one odd
cycle $C$ in $G$ (the cycle of length $5$), such that it is not possible to
extend a normal $5$-coloring of $G$ to the snark $\tilde{G}$ obtained by
superpositioning the edges and vertices on $C$ without changing the colors of
edges outside $C.$
\end{remark}

When this condition $\left.  \tilde{\sigma}\right\vert _{M_{\mathrm{int}}%
}=\left.  \sigma\right\vert _{M_{\mathrm{int}}}$ which cannot hold for
$\tilde{G}$ from Proposition \ref{Prop_odd} is abandoned, then $\tilde{G}$
does have a normal $5$-coloring as it is established by Figure
\ref{Fig_oddSuperpositionColoring} in which $\tilde{\sigma}_{\mathrm{ext}%
}=\left.  \tilde{\sigma}\right\vert _{M_{\mathrm{ext}}}$ is shown, and then
for each $\mathcal{B}_{i}$ a coloring $\tilde{\sigma}_{i}=\left.
\tilde{\sigma}\right\vert _{\mathcal{B}_{i}}$ which is compatible with
$\tilde{\sigma}_{\mathrm{ext}}$ and $\tilde{\sigma}_{i-1}.$

\section{Concluding remarks and further work}

In this paper we considered a snark $\tilde{G}$ obtained from a snark $G$ by
superposing edges of a cycle $C$ in $G$ by a superedge $B$ obtained from
Petersen graph, and vertices of $C$ by one of two simple supervertices $A$ and
$A^{\prime}.$ Assuming that $G$ has a normal $5$-coloring $\sigma$, we
established that in the case of even $C$ the graph $\tilde{G}$ also has a
normal $5$-coloring $\tilde{\sigma}.$ Moreover, $\tilde{\sigma}$ can be
obtained by extending $\sigma$ to $\tilde{G}$ so that the colors of edges of
$G$ outside of the cycle $C$ are preserved. We also observed that in case of
an odd cycle $C$ this is not always possible, i.e. that there exists at least
one graph $G$ and at least one cycle $C$ in $G$ such that it is not possible
to extend a coloring $\sigma$ of $G$ to the superposition $\tilde{G}$ by
preserving colors outside of $C$ in $G.$ It must be noted that we mostly
considered only one particular way of identifying semiedges of superedges and
supervertices in a superposition, and our approach yielded the solution for
one more way of connecting superedges and supervertices as a simple corollary,
all other possible ways of semiedge identification we leave for further work.

\paragraph{Strong coloring.}

An edge-coloring of a cubic graph $G$ in which every edge is rich is called a
\emph{strong coloring}. The smallest number of colors required by a strong
coloring of a cubic graph $G$ is called the \emph{strong chromatic index} and
denoted by $\chi_{s}^{\prime}(G).$ Let $\mathrm{{NC}}(G)$ denote the set of
all normal $5$-colorings of $G$. Petersen Coloring Conjecture is equivalent to
the claim that $\mathrm{{NC}}(G)\not =\emptyset$ for every bridgeless cubic
graph $G.$

Assuming that the Petersen Coloring Conjecture holds i.e. that $\mathrm{{NC}%
}(G)$ is indeed non-empty for every bridgeless cubic graph, we define
$\mathrm{{poor}}(G)$ as the maximum number of poor edges among all colorings
from $\mathrm{{NC}}(G)$. Obviously, if $\mathrm{{poor}}(G)=\left\vert
E(G)\right\vert $ then $G$ is $3$-edge colorable, and if $\mathrm{{poor}%
}(G)=0$ then $\chi_{s}^{\prime}(G)=5.$ Finally, if $0<\mathrm{{poor}%
}(G)<\left\vert E(G)\right\vert \mathrm{,}$ then a relation between
$\chi^{\prime}(G)$ and $\chi_{s}^{\prime}(G)$ is less obvious, it might happen
that $G$ is $3$-edge colorable, with larger strong chromatic index.

The following relation of a strong chromatic index and a covering graph is
given by Lu\v{z}ar et al. \cite{LuzarJGT}. First, a \emph{covering graph}
$\tilde{G}$ of a graph $G$ is any graph for which a surjective graph
homomorphism $f:\tilde{G}\rightarrow G$ exists such that for every vertex
$\tilde{v}$ of $\tilde{G}$ the set of edges incident with $\tilde{v}$ is
bijectively mapped onto the set of edges incident with $f(\tilde{v}).$

\begin{theorem}
The strong chromatic index of a cubic graph $G$ equals $5$ if and only if $G$
is a covering graph of the Petersen graph.
\end{theorem}

\paragraph{Poor edges.}

The above theorem implies that any graph $G$ which is not a covering graph of
the Petersen graph satisfies $\mathrm{{poor}}(G)>0.$ As for a graph $G$ which
is a covering graph of the Petersen graph it might hold $\mathrm{{poor}%
}(G)=0\mathrm{,}$ then every normal $5$-coloring of $G$ is strong. Note that
$G$ might have a normal coloring with at least one poor edge in which case
$\mathrm{{poor}}(G)>0.$

When considering $\mathrm{{poor}}(G)\mathrm{,}$ it is useful to first consider
$3$-cycles and $4$-cycles of a snark $G,$ if $G$ contains such cycles. For
$3$-cycles the following observation is rather well known and obvious.

\begin{remark}
Let $G$ be a bridgeless cubic graph and $\sigma$ a normal $5$-coloring of $G.$
If $G$ contains a $3$-cycle $C$, then every edge of $C$ is poor in $\sigma$.
\end{remark}

A similar observation can be made for $4$-cycles, perhaps it is a folk one too.

\begin{remark}
Let $G$ be a bridgeless cubic graph and $\sigma$ a normal $5$-coloring of $G.$
If $G$ contains a $4$-cycle $C$, then either $2$ or $4$ edges of $C$ are poor
in $\sigma$.
\end{remark}

\begin{proof}
Let $C=u_{0}u_{1}u_{2}u_{3}u_{0}$ be a $4$-cycle in $G,$ assume that the
neighbor of $u_{i}\in V(C)$ which does not belong to $C$ is denoted by
$v_{i}.$ Assume further that edges of $C$ are denoted by $e_{i}=u_{i}u_{i+1}$
and edges incident to $C$ by $f_{i}=u_{i}v_{i}$ for $i=0,\ldots,3.$ The proof
is by contradiction.

Assume first that precisely one edge of $C$ is rich, say $e_{0}$. We may
assume that $\sigma(e_{0})=1,$ and since $\sigma$ is proper also $\sigma
(e_{1})=2$ and $\sigma(f_{1})=3.$ Since $e_{1}$ and $e_{2}$ are poor, it
follows that $\sigma(u_{2})=\sigma(u_{3})=\{1,2,3\},$ which further implies
$\sigma(e_{3})\in\{1,2,3\}$ and that contradicts $e_{0}$ being rich.

Assume now that at least $3$ edges of $C$ are rich, say edges $e_{i}$ for
$i=1,2,3,$ and that $\sigma(e_{i})=i.$ Denote $f_{i}=u_{i}v_{i}$ for
$i=0,\ldots,3.$ Since $e_{2}$ is rich, it follows $\sigma(f_{2})\in\{4,5\},$
we may assume $\sigma(f_{2})=4.$ Then $\sigma(f_{3})=5.$ Since $e_{3}$ is
rich, for $e_{0}$ we have $\sigma(e_{0})\in\{1,2\},$ which contradicts $e_{1}$
being rich.
\end{proof}

We can now conclude that any graph which has a normal $5$-coloring without
poor edges must have girth $\geq5$. Let $P_{10}^{\Delta}$ denote a graph
obtained from $P_{10}$ by truncating one vertex, thus a normal $5$-coloring of
$P_{10}^{\Delta}$ has at least $3$ poor edges and it is easily verified that
it has precisely $3$ poor edges. Our investigation of normal colorings of
small snarks combined with Theorem \ref{Tm_main}$\mathrm{,}$ makes us believe
that the following may hold.

\begin{conjecture}
Let $G$ be a bridgeless cubic graph. If $G\not =P_{10}$, then $\mathrm{{poor}%
}(G)>0$. Moreover, if $G\not =P_{10},P_{10}^{\Delta}$, then $\mathrm{{poor}%
}(G)\geq6.$
\end{conjecture}

\noindent Let us stress that having poor edges in abundance may enable one to
use Kempe chains, which could be the key to solving the Petersen Coloring Conjecture.

\bigskip

\bigskip\noindent\textbf{Acknowledgments.}~~Both authors acknowledge partial
support of the Slovenian research agency ARRS program\ P1-0383 and ARRS
project J1-3002. The first author also the support of Project
KK.01.1.1.02.0027, a project co-financed by the Croatian Government and the
European Union through the European Regional Development Fund - the
Competitiveness and Cohesion Operational Programme.

\end{document}